\documentclass[10pt]{article}

\usepackage[letterpaper, left=1in, top=1in, right=1in, bottom=1in, verbose, ignoremp]{geometry}

\usepackage{url}\RequirePackage[colorlinks,citecolor=blue, linkcolor=blue,urlcolor = blue]{hyperref}
\usepackage{latexsym,amssymb,amsmath,amsfonts,graphicx,color,fancyvrb,amsthm,enumerate,subfigure,mathrsfs,mathtools}
\usepackage[authoryear,round]{natbib}
\usepackage[dvipsnames]{xcolor}
\usepackage{xy}\xyoption{all} \xyoption{poly} \xyoption{knot}
\usepackage{float}
\usepackage{bm}
\usepackage{bbm}
\usepackage{multicol,multirow}
\usepackage{array}
\usepackage{relsize}
\usepackage{chngcntr}
\usepackage{etoolbox}
\usepackage{caption}
\usepackage{tikz}
\usepackage{tikz-cd}
\usetikzlibrary{patterns}
\usepackage{amsopn}
\usepackage{calc}
\usepackage{algorithm}
\usepackage[noend]{algpseudocode}

\usepackage{hyperref}

\newtheorem{theorem}{Theorem}[]
\newtheorem*{theorem*}{Theorem}
\newtheorem{corollary}[theorem]{Corollary}
\newtheorem{lemma}[theorem]{Lemma}
\newtheorem{proposition}[theorem]{Proposition}

\newtheorem*{claim*}{Claim}
\theoremstyle{definition}
\newtheorem{definition}[theorem]{Definition}
\newtheorem*{definition*}{Definition}
\theoremstyle{AppDefinition}

\theoremstyle{AppClaim}

\theoremstyle{remark}
\newtheorem{remark}[theorem]{Remark}

\newtheorem{example}[theorem]{Example}
\newtheorem*{example*}{Example}

\def\beginmat{ \left( \begin{array} }
\def\endmat{ \end{array} \right) }

\makeatletter
\newcommand*{\op}{%
  \DOTSB
  \mathop{\vphantom{\bigoplus}\mathpalette\matt@op\relax}%
  \slimits@
}
\newcommand\matt@op[2]{%
  \vcenter{\m@th\hbox{\resizebox{\widthof{$#1\bigoplus$}}{!}{$\boxplus$}}}%
}
\makeatother

\newcommand{\one}{{\bf{1}}}

\def\R{{\mathbb R}}

\newcommand{\dtr}{d_{\mathrm{tr}}}

\makeatletter
\def\@biblabel#1{}
\makeatother

\makeatletter
\patchcmd{\NAT@citex}
  {\@citea\NAT@hyper@{%
     \NAT@nmfmt{\NAT@nm}%
     \hyper@natlinkbreak{\NAT@aysep\NAT@spacechar}{\@citeb\@extra@b@citeb}%
     \NAT@date}}
  {\@citea\NAT@nmfmt{\NAT@nm}%
   \NAT@aysep\NAT@spacechar\NAT@hyper@{\NAT@date}}{}{}

\patchcmd{\NAT@citex}
  {\@citea\NAT@hyper@{%
     \NAT@nmfmt{\NAT@nm}%
     \hyper@natlinkbreak{\NAT@spacechar\NAT@@open\if*#1*\else#1\NAT@spacechar\fi}%
       {\@citeb\@extra@b@citeb}%
     \NAT@date}}
  {\@citea\NAT@nmfmt{\NAT@nm}%
   \NAT@spacechar\NAT@@open\if*#1*\else#1\NAT@spacechar\fi\NAT@hyper@{\NAT@date}}
  {}{}

\makeatother

\begin{document}
\def\spacingset#1{\renewcommand{\baselinestretch}%
{#1}\small\normalsize} \spacingset{1}
%\spacingset{1.45} % DON'T change the spacing!
\begin{flushleft}
{\Large{\textbf{An Invitation to Tropical Alexandrov Curvature}}}
\newline
\\
{\textbf{Dedicated to Bernd Sturmfels on the occasion of his 60th birthday}}
\newline
\\
Carlos Am\'{e}ndola$^{1,\dagger}$ and Anthea Monod$^{2}$
\\
\bigskip
\bf{1} Institute of Mathematics, Technical University of Berlin, Germany
\\
\bf{2} Department of Mathematics, Imperial College London, UK
\\
\bigskip
$\dagger$ Corresponding e-mail: amendola@math.tu-berlin.de
\end{flushleft}

%%%%%%%%%%%%%%%%%%%%%%%%%%%%%%%%%%%%%%%%%%%%%%%%%%%

\section*{Abstract}

We study Alexandrov curvature in the tropical projective torus with respect to the tropical metric, which has been useful in various statistical analyses, particularly in phylogenomics. Alexandrov curvature is a generalization of classical Riemannian sectional curvature to more general metric spaces; it is determined by a comparison of triangles in an arbitrary metric space to corresponding triangles in Euclidean space.  In the polyhedral setting of tropical geometry, triangles are a combinatorial object, which adds a combinatorial dimension to our analysis. We study the effect that the triangle types have on curvature, and what can be revealed about these types from the curvature. We find that positive, negative, zero, and undefined Alexandrov curvature can exist concurrently in tropical settings and that there is a tight connection between triangle combinatorial type and curvature. Our results are established both by proof and computational experiments, and shed light on the intricate geometry of the tropical projective torus. In this context, we discuss implications for statistical methodologies which admit inherent geometric interpretations.
%We also discuss the implications for statistical methodologies in the context where there are inherent geometric interpretations. %Our study has important implications in statistical methodologies where there are inherent geometric interpretations.

%%%%%%%%%%%%%%%%%%%%%%%%%%%%%%%%%%%%%%%%%%%%%%%%%%%

%Given that many statistical and machine learning methodologies have geometric interpretations, such as linear regression and $k$-means clustering, underlying geometric features of data spaces are important to understand; we study one such feature in this paper, namely, curvature. 

\section{Introduction}
\label{sec:intro}

Several fundamental classical statistical and machine learning methodologies have geometric interpretations, such as linear regression and $k$-means clustering.  The main idea is that underlying geometric features of data spaces are relevant for building and interpreting a model. \emph{Curvature} is one such classical concept in mathematics that has important implications in computational studies and real data applications. 

Likewise, \emph{tropical geometry} is a field that started as a combinatorial shadow of algebraic geometry, studying sets of systems of polynomial equations defined in the tropical semiring $(\mathbb{R} \cup \{\infty\}, \oplus, \odot)$ where for two elements $a,b \in \mathbb{R} \cup \{\infty\}$, $a \oplus b := \min\{a,b\}$, $a \odot b := a + b$, and now has increasing relevance to computational studies and real data applications, such as phylogenomics, economics, and machine learning.

%Tropical geometry deals with piecewise linear and polyhedral structures that arise in the context of algebraic geometry.  In algebraic geometry, the geometry of the zero sets of systems of polynomial equations is studied using commutative algebra; in tropical geometry, these polynomials are defined in the %\footnote{We use the $\min$ convention, preferred by Diane Maclagan and Bernd Sturmfels in \cite{maclagan2015introduction}.} 
%and $a \odot b := a + b$.  Tropical geometry is also by nature combinatorial; it is a subject that combines algebraic and polyhedral geometry and combinatorics.  In this paper, we present a study of a combinatorial aspect of tropical geometry and show that it has surprising implications in the metric geometric concept of curvature.
%Given the , studying curvature in tropical geometry is therefore also important.  

With this statistical motivation in mind, in this paper we combine curvature with tropical geometry and study \emph{Alexandrov curvature}---a general measure for arbitrary metric spaces that depends on geodesic triangles---in the tropical projective torus. In tropical geometry, there are combinatorially different kinds of triangles. It turns out that the combinatorial type of tropical triangle can determine the curvature; and in general, the curvature behavior is tightly connected to this type.  In this sense, the interest of our study is at least twofold: it is the first study of Alexandrov curvature in tropical geometry and, to the best of our knowledge, it is also the first instance where Alexandrov curvature is computed on a combinatorial object.

In this work:
\begin{itemize}
\item We show that Alexandrov curvature in the plane with respect to the tropical metric may be positive, negative, and undefined in the same space: in particular, we prove that two out of the five combinatorial triangle types in the plane are always undefined, and one always exhibits positive curvature;
\item We give proportions of the occurrence of the two remaining combinatorial types of tropical triangle by random sampling: we show that one type is almost always negatively curved, while the other is almost always undefined;
\item We give a procedure to randomly sample tropical triangles by type;
\item We explore tropical triangles and curvature in higher dimensions, in particular, we study by simulation the Alexandrov curvature behavior on sets of phylogenetic trees; and finally,
\item We provide an open repository that makes all the code for all our numerical experiments publicly available.
\end{itemize}
These results are relevant because they have direct implications on statistical studies and computation in tropical geometric settings.  Additionally, they are interesting because they reveal the intricate and complex geometric nature of the tropical projective torus, indicating that there still remains much to be understood.

The remainder of this paper is organized as follows.  We first provide explicit statistical motivations for our study in Section \ref{sec:motivation} along with an overview on existing metric geometric work in tropical geometry.  In Section \ref{sec:tropical_metric}, we introduce the tropical projective torus and tropical metric as our space of interest.  In Section \ref{sec:curvature}, we give a brief overview of the general mathematical notion of curvature and discuss previous studies of curvature in tropical settings.  In Section \ref{sec:alexandrov_plusminus}, we study positive and negative Alexandrov curvature in the tropical projective torus and give our first theoretical and experimental results.  We then turn to undefined Alexandrov curvature in the plane in Section~\ref{sec:undefined} and numerically explore the role that the combinatorial type of a triangle plays in tropical curvature.  We extend our study to higher dimensions in Section \ref{sec:higherdim} and include an experiment on the curvature behavior on sets of phylogenetic trees.  We end with a discussion in Section \ref{sec:discussion} on future directions for study and the potential implications of tropical Alexandrov curvature on other computational works in tropical geometry.

%%%%%%%%%%%%%%%%%%%%%%%%%%%%%%%%%%%%%%%%%%%%%%%%%%%

\section{Motivation: Tropical Statistics and Metric Geometry}
\label{sec:motivation}

Our main motivation for studying the curvature of the tropical projective torus is statistical.  There is an intricate connection between tropical geometry and the space of phylogenetic trees, which is a groundbreaking result that was established by David Speyer and Bernd Sturmfels in 2004 \citep{Speyer2004}.  This theoretical result has recently gained much interest from computational and applied mathematics researchers as an alternative setting for statistical analyses of sets of phylogenetic trees because such analyses with current tools are extremely difficult.  Being able to efficiently analyze sets of phylogenetic trees in an interpretable manner is crucial, particularly in the face of a global pandemic.  Recent work has shown that the tropical geometric perspective is compatible with formal probability and statistics; moreover, statistical techniques have even been developed on tree space as well as the tropical projective torus as its ambient space with promising results for increased performance over current standards \citep{doi:10.1137/16M1079841,monod2018tropical,yoshida2019tropical,math9070779,yoshida2023tropical}.  Other occurrences of tropical geometry in the general context of algebraic statistics for computational biology arise in the book written by Lior Pachter and Bernd Sturmfels \citep{ASCB}.  Tropical geometry has even found its way to machine and deep learning, e.g., in \cite{9394420,pmlr-v80-zhang18i,trimmel2021tropex}.

In statistics, one of the most essential quantities is a measure of central tendency of data, which is a value that is as close as possible to all points in a dataset.  In arbitrary metric spaces, this measure of central tendency is called a {\em Fr\'{e}chet mean}.  One of the most classical algorithms to compute Fr\'{e}chet means is Sturm's algorithm \citep{sturm2003probability}.  Sturm's algorithm relies on the metric space being nonpositively curved.  Therefore, understanding the curvature of a space is important in order to understand if and where Sturm's algorithm may or may not be applied to compute Fr\'{e}chet means in the tropical projective torus---or, more specifically, in the space of phylogenetic trees).  Recent research in statistics has shown that it is becoming increasingly important to take into account geometric features when developing statistical methodology on nonstandard spaces \citep{https://doi.org/10.1002/wics.1526,Kobayashi2019}.

Additionally, the curvature of a space is a fundamental invariant and an important metric geometric feature that is interesting to study in itself.  Our study of curvature adds to the growing body of work on metric geometry in tropical geometry.  Recently, the intrinsic geometry of tropical semialgebraic sets is significant in work by Alessandrini and Jell and co-authors \citep{Alessandrini+2013+155+190,10.1093/imrn/rnaa112}; this is relevant because some instances of curvature are also intrinsic (please see the discussion to come in Section \ref{sec:curvature}).  In addition, metric geometric aspects such as volume, bisectors, and Voronoi diagrams have been introduced by Loho and Schymura, and Criado and co-authors \citep{Loho2020,criado2021tropical}.  Notably, Bo Lin, Bernd Sturmfels and co-authors study convexity in tree spaces using the tropical metric \citep{doi:10.1137/16M1079841}.  Recently, an interest in differential geometric and probabilistic aspects has resulted in a study of optimal transport and Wasserstein distances in the tropical projective torus \citep{Lee2021}.

%%%%%%%%%%%%%%%%%%%%%%%%%%%%%%%%%%%%%%%%%%%%%%%%%%%

\section{The Tropical Projective Torus, Tropical Metric, and Tropical Line Segments} %on $\R^n \xrightarrow{\sim} \R^n/\R\one \xhookrightarrow{\cong} \R^{n-1}$}
\label{sec:tropical_metric}

%Many concepts and questions in classical algebraic geometry may be reinterpreted and studied in the tropical setting with interesting and relevant parallels.  One such example is Gr\"{o}bner bases, which are special generating sets of ideals in a polynomial ring over a field, and a fundamental tool in solving systems of polynomial equations. Reinterpreting Gr\"{o}bner bases using valuations generates Gr\"{o}bner complexes and universal Gr\"{o}bner bases, which are analogously related to tropical bases \citep{maclagan2015introduction}. The Gr\"{o}bner complex is a polyhedral complex for a homogeneous ideal in the polynomial ring $K[x_1, \ldots, x_n]$ over a field $K$; its ambient space is the {\em tropical projective torus}.  
The tropical projective torus is the space in which we work in this paper.

\begin{definition}
For $x, y \in \mathbb{R}^n$, consider the equivalence relation 
$$
x \sim y \Leftrightarrow \mbox{all coordinates of~} (x-y) \mbox{~are equal.}  
$$
The {\em tropical projective torus} $\mathbb{R}^n/\mathbb{R}\one$ is the quotient space given by the set of equivalence classes under $\sim$.
\end{definition}

\begin{remark}
\label{rem:vectorspace}
The notation $\R^n/\R\one$ can literally be interpreted as a quotient vector space, namely Euclidean space modulo the subspace generated by the span of the all-ones vector $\one = (1,\dots,1)$.
\end{remark}

%Notice that the tropical projective torus is the space that is constructed by identifying vectors that differ from each other by tropical scalar multiplication.  

%The tropical projective torus may also be characterized by a group action \cite{monod2018tropical}: let $G := \{(c,\ldots, c) \in \R^n \mid c \in \R\}$ with coordinate-wise addition.  $G$ is an additive group that acts on $\R^n$ by $g \circ x = (x_1 + g_1,\, x_2 + g_2,\, \ldots,\, x_n + g_n)$ for $g \in G$ and $x \in \R^n$.  Each point in $\R^n/\R\one$ is exactly one orbit under the action of $G$ on $\R^n$.

%The embeddability of the tropical projective torus into Euclidean space makes this important space particularly compatible with computational studies: $\R^n/\R\one$ identifies with $\R^{n-1}$ by taking representatives of the equivalence classes where the first coordinate vanishes, i.e., $(x_2 - x_1,\, x_3 - x_1,\, \ldots,\, x_n - x_1) \in \R^{n-1}$.

The tropical projective torus may be equipped with a metric, giving rise to a metric space.  Our metric of interest in this paper is the {\em tropical metric}.  As its name implies, the tropical metric arises in the context of tropical geometry and has been referred to as a generalized projective Hilbert metric in other literature \citep{AKIAN20113261,COHEN2004395}.  In tropical geometry, it has been used in tree settings, convexity studies, and probability \citep{joswig2007josephine,hampe2015tropical,doi:10.1137/16M1079841,monod2018tropical,tran2020tropical,yoshida2019tropical}.

\begin{definition}
\label{def:tropicalmetric}
Let $x, y \in \mathbb{R}^n$ and let $[x],[y]$ be their representatives in the tropical projective torus.  We define the tropical metric on $\mathbb{R}^n/\mathbb{R}\one$ as
\begin{align*}
\dtr([x],[y]) :=& \max_{1\leq i < j \leq {n}}\big|(x_i-y_i)-(x_j -y_j)\big|\\
=& \max_{1 \leq i \leq n} (x_i - y_i) - \min_{1 \leq i \leq n} (x_i - y_i).
\end{align*}
\end{definition}
%The tropical metric is a rigorous and well-defined metric.

The metric space $(\R^n/\R\one, \dtr)$ can be identified with a normed linear space in the following manner.  Consider the following map 
\begin{align*}
\pi: \R^n/\R\one & \rightarrow \R^{n-1}\\
[x] & \mapsto (x_2 - x_1, \ldots, x_n - x_1); 
\end{align*}
$\pi$ is a linear isomorphism.  We may define a norm on $\R^{n-1}$ by 
$$
\|x\|_{\mathrm{tr}} := \max \left( \max_{1 \leq i < j \leq n}|x_i - x_j|,\, \max_{1 \leq i \leq n}|x_i| \right);
$$
denote the induced distance by $\hat{d}_{\mathrm{tr}}$.  Then
$$
\dtr([x], [y]) = \|\pi([x]) - \pi([y])\|_{\mathrm{tr}} = \hat{d}_{\mathrm{tr}}(\pi([x]), \pi([y]))
$$
%\begin{align*}
%\dtr([x], [y]) & = \max\bigg( \max_{2 \leq i < j \leq n} |(x_i - y_i) - (x_j - y_j)|,\, \max_{2 \leq i \leq n} |x_i - y_i| \bigg)\\
%& = \|\pi([x]) - \pi([y])\|_{\mathrm{tr}} = \hat{d}_{\mathrm{tr}}(\pi([x]), \pi([y]))
%\end{align*}
and $\pi$ is an isometry. The isometric embedding of the tropical projective torus into Euclidean space $\pi$ makes $(\R^n/\R\one, \dtr)$ particularly compatible with computational studies. However, we have the following fact.

%Take representatives of $x$ and $y$ (call these $\bar x$ and $\bar y$), where the first coordinate vanishes and map all other coordinates to $\mathbb{R}^{n-1}$ to obtain an embedding of $\mathbb{R}^n/\mathbb{R}\one$ into $\mathbb{R}^{n-1}$:
%$$
%x \mapsto \bar x := (x_2 - x_1,\, x_3 - x_1,\, \ldots, \, x_n - x_1).
%$$

%Then the tropical metric on $\mathbb{R}^n$ translates to the following on $\mathbb{R}^{n-1}$:
%$$
%\dtr(\bar x, \bar y) := \max\Big\{\max_{1\leq i < j \leq n-1}\big|(\bar x_i - \bar y_i)-(\bar x_j - \bar y_j)\big|,\, \max_{1\leq i \leq n-1}|\bar x_i - \bar y_i|\Big\}.
%$$

%$\mathbb{R}^n$ identifies with $\mathbb{R}^n/\one$ by the equivalence relation $\sim$; $\R^n/\one$ then embeds into $\mathbb{R}^{n-1}$.  The metric $\dtr$ is defined on $\R^n/\one$ and has a representation in $\mathbb{R}^{n-1}$, its representation from $\R^n$ to $\R^n/\one$ to $\R^{n-1}$ must be isometric. 
%$$
%\begin{tikzcd}
%\R^n \arrow[r, "\sim"] \arrow[urrd, bend right, leftrightarrow, dashed] & \R^n/\one \arrow[r, hook, "\cong"] & \R^{n-1}
%\end{tikzcd}
%$$

%\begin{proposition}
%\label{prop:catk}
%Palm tree space is not a $\mathrm{CAT}(k)$ space for any $k > 0$.
%\end{proposition}

%\begin{proof}
%If palm tree space were a $\mathrm{CAT}(k)$ space, then for all $x, y$ in palm tree space such that $d_{\mathrm{tr}} < \pi^2/k$, there is a unique geodesic between $x$ and $y$.  But there exist $x$ and $y$ arbitrarily close in palm tree space with two distinct geodesics, then palm tree space cannot be a $\mathrm{CAT}(k)$ space.
%\end{proof}

\begin{proposition}
\label{prop:parallelogram}
For $n \geq 3$, $(\R^{n-1}, \| \cdot \|_{\mathrm{tr}})$ is not a Hilbert space.
\end{proposition}

\begin{proof}
Recall that in a normed space $(V, \|\cdot\|)$, if there is an inner product on $V$ such that $\|x\|^2 = \langle x,x \rangle$ for all $x \in V$, then the parallelogram law
$$
\|x+y\|^2 + \|x-y\|^2 = 2\|x\|^2 + 2\|y\|^2
$$
must hold. However, consider
$x = (1,0, \ldots, 0)$ and $y= (0,1,0,\ldots,0)$.  Then
$$
\|x\|_{\mathrm{tr}} = \|y\|_{\mathrm{tr}} = 1 \qquad \mbox{~and~} \qquad \|x+y\|_{\mathrm{tr}} = 1,\, \|x - y\|_{\mathrm{tr}}=2,
$$
but $1^2 + 2^2 \neq 2(1^2 + 1^2)$, so the parallelogram law does not hold.
\end{proof}

\begin{corollary}
\label{cor:cat}
$(\R^{n-1}, \hat{d}_{\mathrm{tr}})$ is not a $\mathrm{CAT}(k)$ space for any $n \geq 3, k \in \R$.
\end{corollary}

\begin{proof}
By Proposition 1.14 of \cite{bridson2013metric}, a normed linear space is a $\mathrm{CAT}(k)$ space for $k \in \R$ if and only if the norm is induced by an inner product.  Since $(\R^{n-1}, \hat{d}_{\mathrm{tr}})$ is not a Hilbert space, it is not $\mathrm{CAT}(k)$ for any $k$.
\end{proof}
This corollary has an immediate implication on the curvature of the tropical projective torus endowed with the tropical metric and rules out the vast literature of results on $\mathrm{CAT}(k)$ spaces, e.g., \cite{jost2012nonpositive,ohta2012barycenters}.  In particular, we know that that $(\R^{n-1}, \hat{d})$ is not a $\mathrm{CAT}(0)$ space where geodesics are unique.  In fact, there are infinitely many tropical geodesics between any two points in $(\R^n/\R\one, \dtr)$, e.g., \cite{monod2018tropical}.  This leads naturally to consider the following object.

\begin{definition}
\label{def:tls}
The {\em tropical line segment} connecting $[x], [y] \in \R^n/\R\one$ is the set
$$
\gamma_{xy} = \{\alpha \odot [x] \oplus \beta \odot [y] \mid \alpha,\beta \in \R\},
$$
where tropical addition is performed coordinate-wise.
\end{definition}

%When working with arbitrary points in the tropical projective torus rather than particular representatives of an equivalence class, we may write $x \in \R^n/\R\one$, rather than $[x]$.  Since we are interested in studying curvature behavior using objects defined by points in this paper, we write simply $x$, as above in Definition \ref{def:tls} and from now on.

The tropical line segment between any two points in $\R^n/\R\one$ is unique and it is a geodesic \citep{monod2018tropical}.  We will often reparametrize the tropical line segment as
\begin{equation}
\label{eq:tls}
\gamma_{xy}(t) := t \odot [x] \oplus [y]
\end{equation}
where $t:= \alpha-\beta$.  In other words, $\gamma_{xy}(t)$ is the tropical line segment connecting $[x]$ to $[y]$, parametrized by $t \in \R$.  Note that it is sufficient for the parametrization $t$ to fall within $[0, \dtr(x,y)]$, which we use in the examples in this paper.

\begin{remark}\label{rem:notation}
To ease notation, from now on, when we write $x=(x_1,\dots,x_n) \in \R^{n+1}/\R\one$ we will refer to the equivalence class $[x]=[(0,x_1,\dots,x_n)]$. This way we can apply directly $\dtr$ to such points remembering the initial zero. The next example illustrates this.
\end{remark}

\begin{example}
\label{ex:tls_1}
Fix points $b = (0,0)$ and $c = (3,2)$ in $\R^3/\R\one$. We have that 
\[
\dtr(b,c) = \max(0, 0-3, 0-2) - \min(0, 0-3, 0-2) = 0 - (-3) = 3.
\]
The tropical line segment $\gamma_{bc}(t)$ for $t \in [0,3]$ connecting $b$ and $c$ is given by
\begin{align}
\gamma_{bc}(t) & = t \odot [(0,0,0)] \oplus [(0,3,2)] \nonumber \\ 
& = \min((t+0,0),\, (t+0,3),\, (t+0,2)) \nonumber \\
%& = (\min(t,0),\, \min(t,3),\, \min(t,2)) \nonumber \\
& = \begin{cases} 
[(0,t,t)], & 0 \leq t < 2; \\
[(0,t,2)], & 2 \leq t \leq 3. \label{eq:tls_ex1} 
\end{cases}
\end{align}
The tropical line segment takes the form as in Figure \ref{fig:tls_types}(a).
\end{example}

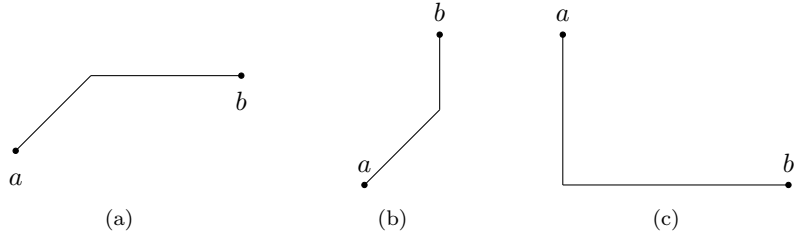
\begin{figure}
\centering
\subfigure[]{
\begin{tikzpicture}
\coordinate (a) at (0,0);
\coordinate (b) at (3,1);
\coordinate (k) at (1,1);
\foreach \coord in {a,b}
	\draw[fill=black] (\coord) circle[radius=1pt];
%\foreach \coord in {a,b}
%	\node[label={[label distance=-2pt]90:$\coord$}]at (\coord) {};
\foreach \coord in {a,b}
\node[label={[label distance=-20pt]:$\coord$}] at (\coord) {};
\draw (a)--(k);
\draw (b)--(k);
\end{tikzpicture}}
\hskip1cm
\subfigure[]{
\begin{tikzpicture}
\coordinate (a) at (0,0);
\coordinate (b) at (1,2);
\coordinate (k) at (1,1);
\foreach \coord in {a,b}
	\draw[fill=black] (\coord) circle[radius=1pt];
\foreach \coord in {a,b}
	\node[label={[label distance=-2pt]90:$\coord$}]at (\coord) {};
\draw (a)--(k);
\draw (b)--(k);
\end{tikzpicture}}
\hskip1cm
\subfigure[]{
\begin{tikzpicture}
\coordinate (a) at (0,2);
\coordinate (b) at (3,0);
\coordinate (k) at (0,0);
\foreach \coord in {a,b}
	\draw[fill=black] (\coord) circle[radius=1pt];
\foreach \coord in {a,b}
	\node[label={[label distance=-2pt]90:$\coord$}]at (\coord) {};
\draw (a)--(k);
\draw (b)--(k);
\end{tikzpicture}}
\caption{Types of Tropical Line Segment: (a) L1; (b) L2; (c) L3.}
\label{fig:tls_types}
\end{figure}

In general, a tropical line segment in $\R^n/\R\one$ consists of a concatenation of $n-1$ Euclidean line segments (see Proposition 3 of \cite{develin2004tropical} and Remark 2.4 in \cite{hampe2015tropical}). Since we will study carefully the plane case, we record here the following explicit description for $n=3$.

\begin{lemma}
\label{lemma:tls_types}
There are three types of tropical line segments between two points in general position $a = (a_1, a_2)$ and $b = (b_1, b_2)$ in $\R^3/\R\one$; w.l.o.g. $a_1 < b_1$.  Each tropical line segment is characterized by inequalities on the coordinates and exhibits a unique bending point; the endpoints $a$, $b$ and the bending point are joined by lines of slope 0, 1, or $\infty$.  Their explicit forms are given as follows:
\begin{enumerate}[L1]
\item For $a_2 < b_2$ and $a_1-a_2 < b_1 - b_2$, the bending point is at $(a_1 + b_2 - a_2, b_2)$, $t = b_2 - a_2$, and
\begin{equation*}
\label{eq:tls_l1}
\gamma_{ab}(t) = \begin{cases}
(a_1 + t, a_2 + t), & 0 \leq t \leq b_2 - a_2;\\
(a_1 + t, b_2), & b_2 - a_2 \leq t \leq b_1 - a_1.
\end{cases}
\end{equation*}
Its length is $b_1 - a_1$.

\item For $a_2 < b_2$ and $a_1-a_2 > b_1 - b_2$, the bending point is at $(b_1, a_2 + b_1 - a_1)$, $t = b_1 - a_1$, and 
\begin{equation*}
\label{eq:tls_l2}
\gamma_{ab}(t) = \begin{cases}
(a_1 + t, a_2 + t), & 0 \leq t \leq b_1 - a_1;\\
(b_1, a_2 + t), & b_1 - a_1 \leq t \leq b_2 - a_2.
\end{cases}
\end{equation*}
Its length is $b_2 - a_2$.

\item For $a_2 > b_2$, the bending point is at $(a_1, b_2)$, $t = a_2 - b_2$, and 
\begin{equation*}
\label{eq:tls_l3}
\gamma_{ab}(t) = \begin{cases}
(a_1, a_2 - t), & 0 \leq t \leq a_2 - b_2;\\
(a_1 + b_2 - a_2 + t, b_2), & a_2 - b_2 \leq t \leq a_2 - b_2 + b_1 - a_1.
\end{cases}
\end{equation*}
Its length is $(a_2 - b_2) + (b_1 - a_1)$.
\end{enumerate}
\end{lemma}

\begin{proof}
Consider type L1: from the defining inequality $a_1 - a_2 < b_1 - b_2$, we have that $a_1 - b_1 < a_2 - b_2$.  We use this to find the length of this line segment by computing
\begin{align}
\dtr(a,b) & = \max(0,\, a_1 - b_1,\, a_2 - b_2) - \min(0,\, a_1 - b_1,\, a_2 - b_2) \label{eq:tls_length}\\
 & = 0 - (a_1 - b_1)  = b_1 - a_1. \nonumber
\end{align}
To compute the tropical line segment between $a$ and $b$, we use (\ref{eq:tls}) and the defining inequalities $a_1 < b_1$ and $a_2 < b_2$ to compute
\begin{align*}
\gamma_{ab}(t) & = (\min(t, 0),\, \min(a_1 + t, b_1),\,  \min(a_2 + t, b_2)) \\ \nonumber %\label{eq:tls_expanded} \\ 
& = (0,\, a_1 + t,\, a_2 + t), \nonumber
\end{align*}
as long as $t \leq b_1 - a_1$.  For $t \geq b_1 - a_1$, we still have $a_1 + t \leq b_1$ but now $b_2 \leq a_2 + t$, so
$$
\gamma_{ab}(t) = (0,\, \min(a_1 + t, b_1),\, \min(a_2 + t, b_2)) = (0,\, a_1 + t,\, b_2),
$$
giving, as desired,
$$
\gamma_{ab}(t) = \begin{cases}
(a_1 + t, a_2 + t), & 0 \leq t \leq b_2 - a_2;\\
(a_1 + t, b_2), & b_2 - a_2 \leq t \leq b_1 - a_1.
\end{cases}
$$

A similar computation with the defining inequality $a_1 - a_2 > b_1 - b_2$ for L2 gives the desired result, by noting that here, the direction of the inequality is reversed to that of L1.

For L3, we compute (\ref{eq:tls_length}) and note that we now have $a_2 > b_2$ while $a_1 < b_1$, so
$$
\dtr(a,b) = a_2 - b_2 - (a_1 - b_1) = (a_2 - b_2) + (b_1 - a_1).
$$
An analogous computation as above gives the desired result for $\gamma_{ab}(t)$.
\end{proof}

%%%%%%%%%%%%%%%%%%%%%%%%%%%%%%%%%%%%%%%%%%%%%%%%%%%

\section{Curvature: Classical and Alexandrov}
\label{sec:curvature}

Curvature is a deep mathematical concept that has been extensively studied in various settings; it is fundamental in Riemannian and, more generally, differential geometry.  Here, we provide a very brief whistle-stop tour of some classical notions of curvature. It serves our twofold purposes of motivating our study of Alexandrov curvature and mentioning previous studies of curvature in tropical geometry.

Classical and vastly more comprehensive discourses on the subject of curvature have been documented by, for example, Burago and Ivanov; Saucan; and Gromov \citep{burago2001course,saucan2006curvature,gromov2019four}.

\subsection{Curvature for Curves}
\label{subsec:curvature_curves}

In the simplest case of a circle, its curvature is given by $1/r$ where $r$ is the radius.  This notion can be extended to plane curves where the curvature of a plane curve at the point $p$ is the curvature of the circle that best fits the curve at the point $p$.  More precisely, it can be thought of as the curvature of the largest circle that is fully contained on one side of the curve and has one common point with $p$.  Although this concept is rigorous, the main challenge is computational: the {\em osculating circle} described by Newton in 1687 overcomes this limitation by defining the limit of circles with three common points with the curve \citep{newton1833philosophiae}.  If the image of a function $f: [0,1] \rightarrow \mathbb{R}^2$ gives the path $\gamma$, then the osculating circle $C$ at $\gamma_0 = f(t_0)$ is defined by
$$
C(\gamma_0) = \lim_{\gamma_1, \gamma_2 \rightarrow \gamma_0} C(\gamma_0, \gamma_1, \gamma_2) = \lim_{t_1, t_2 \rightarrow t_0} C(t_0, t_1, t_2),
$$
where $\gamma_i = \gamma(t_i)$ for $i=1,2$.  The {\em curvature} of $\gamma$ at $\gamma_0$ is then defined as $\kappa_\gamma(\gamma_0) := 1/r(C(\gamma_0))$, where $r(C(\gamma_0))$ is the radius of $C(\gamma_0)$.  The osculating circle is compatible with the theoretical notion of maximality for circles given above; see, e.g., \cite{spivak1973comprehensive}.

Curvature of curves and paths has been previously studied in tropical geometry in the context of optimization theory and linear programming \citep{allamigeon2018log}.  Specifically, in the context of interior-point methods---a class of algorithms solving linear and nonlinear convex optimization problems---the {\em total curvature} of a specific curve known as the central path gives a measure of computational complexity.  The total curvature of a curve is the integral of the norm of the acceleration, assuming travel over the curve at unit speed.  Using tropical geometry, Allamigeon and co-authors were able to disprove a 30-year-old conjecture on an upper bound for the computational complexity.  To do this, they tropicalize the central path and  approximate the classical total curvature of the tropicalized central path by polygonal curves, thus giving lower bounds on the classical total curvature of the classical central path (see Theorem 25 of \cite{allamigeon2018log}).  They also give an approximation of the tropicalized central path itself by tropical line segments.

We point out here that the approach used in the work of Allamigeon and co-authors is in the same spirit of our work: we are similarly studying a classical concept of a tropical geometric object.  In their work, they are studying the classical total curvature of the tropicalized central path (which they refer to as ``tropical total curvature'' in the introduction of their paper); similarly, in this paper, we aim to study the classical notion of Alexandrov curvature of tropical triangles in the tropical projective torus (which we similarly refer to as ``tropical Alexandrov curvature'' in the title of our paper).  In particular, we are not tropicalizing Alexandrov curvature, in contrast to a large body of existing work that tropicalizes classical concepts---for instance, by replacing classical operators in a particular object by their tropical equivalents, e.g., \cite{kolokoltsov2013idempotent,tran2020tropical}.  %This approach is in contrast to ours, and hence an entirely different study.

\subsection{Alexandrov Curvature}

Alexandrov curvature is a generalization of sectional curvature (i.e., the higher-dimensional analogue of Gaussian curvature), which relies on an inner product structure.  By Proposition \ref{prop:parallelogram} above, $(\R^n/\R\one, \dtr)$ does not comprise an inner product structure via its isometric embedding into $(\R^{n-1}, \hat{d}_{\mathrm{tr}})$, necessitating a more flexible notion of curvature.  Alexandrov curvature is the natural answer and direct extension \citep{ollivier2011visual,saucan2006curvature}.  Moreover, it is an intrinsic geometric property, which adds to other existing results on intrinsic geometry in tropical geometry \citep{Alessandrini+2013+155+190,10.1093/imrn/rnaa112,Loho2020}.

Let $(X, d_X)$ be a geodesic space, then a {\em triangle} in $X$ is given by three points $(a,b,c) \in X^3$, known as the {\em vertices} of triangle, together with three geodesic curves from $a$ to $b$, $b$ to $c$, and $c$ to $a$, known as the {\em edges} or {\em sides} of the triangle, with lengths of these curves realizing the distances $d_X(a,b)$, $d_X(b,c)$, and $d_X(c,a)$, respectively.  The {\em curvature criterion of Alexandrov} states that triangles become ``skinnier" under negative curvature, and ``fatter" under positive curvature; see Figure \ref{fig:ollivier_skinny} for an example of a skinny triangle.  ``Skinniness" and ``fatness" are measured by the distance between a vertex of the triangle and any point on its opposite edge in comparison to the Euclidean counterpart: for any triangle in $X$, there exists a {\em comparison triangle} $(a', b', c')$ in the Euclidean plane whose sides have the same lengths as its counterpart in $X$.  Such a comparison triangle is unique up to isometry.  Notice that Alexandrov is directly compatible in philosophy to Gaussian curvature, previously discussed above in Section \ref{subsec:curvature_curves}; in this sense, Alexandrov curvature is at the same time classical and general.

\begin{figure}[h!]
\centering
\includegraphics[scale=0.65]{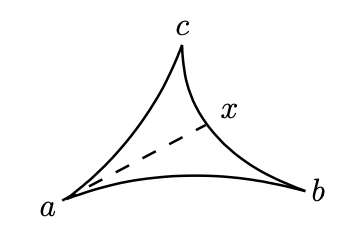}
\caption{An example of a skinny triangle in the sense of Alexandrov \citep{ollivier2011visual}.}
\label{fig:ollivier_skinny}
\end{figure}

\begin{definition}{\cite[Definition 2.1]{ollivier2011visual}}
\label{def:alexandrov}
Let $(X, d_X)$ be a geodesic space; let $d_e$ denote the usual Euclidean metric.  $X$ is said to be a space of {\em curvature $\leq 0$ in the sense of Alexandrov} (or the space has negative (nonpositive) Alexandrov curvature) if, for any small enough triangle $(a,b,c)$ in $X$, and for any point $x$ on the $bc$ edge of this triangle, the following holds:
\begin{equation}
\label{eq:alexandrovneg}
d_X(a,x) \leq d_e (a', x'),
\end{equation}
where $(a', b', c')$ is the comparison triangle of $(a, b, c)$ in the (flat, 0-curvature) Euclidean plane, and $x'$ is a point on the $b' c'$ side corresponding to $x$; i.e., such that $d_e(x', b') = d_X(x, b)$.  Similarly, $X$ is a space of {\em curvature $\geq 0$ in the sense of Alexandrov} (or the space has positive (nonnegative) Alexandrov curvature) if, in the same situation, the following reverse inequality holds:
\begin{equation}
\label{eq:alexandrovpos}
d_X(a, x) \geq d_e(a', x').
\end{equation}
\end{definition}

We will say that a triangle is \emph{skinny} or has \emph{negative curvature} if it satisfies inequality \eqref{eq:alexandrovneg}, and is \emph{fat} or has \emph{positive curvature} if it satisfies inequality \eqref{eq:alexandrovpos}. If equality holds in \eqref{eq:alexandrovneg} and \eqref{eq:alexandrovpos} we say that the triangle is \emph{flat} or has \emph{zero curvature}.

\paragraph{Triangle Sizes.}

In a $\mathrm{CAT}(k)$ space for $k \leq 0$ (i.e., for curvature less than or equal to 0), {\em all} triangles, not only small ones, satisfy the comparison criterion (\ref{eq:alexandrovneg}) under equality.  This fact motivates us to establish a similar result in the tropical setting, as follows.

%Immediately, we have the following observation.

%\begin{proposition}
%\label{prop:nonneg_alexandrov}
%$(\R^n/\R\one, \dtr)$ is not a non-negatively curved Alexandrov space.
%\end{proposition}

%\begin{proof}
%Given $x$ and $y$ in a geodesic space $\mathcal{X}$ endowed with metric $d$, and any geodesic $\gamma: [0,1] \rightarrow \mathcal{X}$, and any $z \in \mathcal{X}$, the following inequality characterizes Alexandrov spaces with curvature bounded from below by zero:
%\begin{equation}
%\label{eq:alexandrov}
%d\big( z, \gamma(t) \big)^2 \geq t \cdot d(z, y)^2 + (1-t)\cdot d(z,x)^2 - t(1-t) \cdot d(x,y)^2.
%\end{equation}
%The inequality (\ref{eq:alexandrov}) does not hold for $(\R^n/\R\one, \dtr)$.  Consider the following counterexample: let $n=3$, and let $x = (-1, 0, 0),\, y = (0, -1, 0),\, z = (0, 0, 2)$.  Then we have $d_{\mathrm{tr}}(x,y) = 2$, and $d_{\mathrm{tr}}(x, z) = d_{\mathrm{tr}}(y, z) = 3$.  For $t \leq 1/2$, we have $\gamma(t) = (2t-1, 0, 0)$ and $d_{\mathrm{tr}}(z, \gamma(t)) = 3-2t$.  Now, (\ref{eq:alexandrov}) becomes
%\begin{align*}
%(3-2t)^2 & \geq t\cdot 3^2 + (1-t)\cdot 3^2 - t(1-t)\cdot 2^2\\
%4t^2 - 12t + 9 + 4t - 4t^2 & \geq 9\\
%0 & \geq 8t,
%\end{align*}
%a contradiction.  Hence, $(\R^n/\R\one)$ is not a non-negatively curved Alexandrov space.
%\end{proof}

\begin{lemma}
\label{lem:scale}
Consider the triangle $(a,b,c) \in (\R^n/\R\one)^3$ and its scaled version $(\alpha a, \alpha b, \alpha c)$ by some $\alpha \in \R_{\geq 0}$. Then the triangles share the same Alexandrov curvature.
\end{lemma}

\begin{proof}
Let $a=[(a_1,\dots,a_n)]$, $b=[(b_1,\dots,b_n)]$ and the corresponding vertices $a',b'$ of the Euclidean comparison triangle.

%In $\R^n/\R\one$, a triangle $(a,b,c)$ scaled by $\alpha$ is defined by the points $(\alpha a, \alpha b, \alpha c)$.  
The length of the edge of the scaled triangle is given by
\begin{align*}
\dtr(\alpha a, \alpha b) & = \max_{1 \leq i \leq n}(\alpha a_i - \alpha b_i) - \min_{1 \leq i \leq n}(\alpha a_i - \alpha b_i)\\
%& = \max_{1 \leq i \leq n}\big(\alpha(a_i - b_i) \big) - \min_{1 \leq i \leq n}\big( \alpha(a_i - b_i) \big)\\
& = \alpha \max_{1 \leq i \leq n} (a_i - b_i) - \alpha \min_{1 \leq i \leq n}(a_i - b_i)\\
& = \alpha \cdot \dtr(a, b).
\end{align*}
A verbatim calculation gives the same result for the lengths of the other two edges of the triangle.  By the same argument, the tropical distance from any vertex to its opposing edge is scaled by $\alpha$.  Since the Euclidean comparison triangle also scales in the same manner, i.e., $d_e(\alpha a', \alpha b') = \alpha \cdot d_e(a',b')$, the direction of either inequality (\ref{eq:alexandrovneg}) or (\ref{eq:alexandrovpos}) is preserved. Hence the Alexandrov curvature remains invariant: it remains negative, positive or flat.
\end{proof}

\begin{remark}
Scaling the vertices in Lemma~\ref{lemma:tls_types} also scales the domain of the length parametrization. If we want to preserve the domain, we need to speed up or slow down the parametrization by taking the parameter $t'=\alpha t$.
\end{remark}

An important consequence of Lemma \ref{lem:scale} is that finding a ``small enough" tropical triangle to study Alexandrov curvature becomes unnecessary: the size of the triangle is not important, and for a given tropical triangle, we may choose any copy of it to study curvature in $\R^n/\R\one$.  

\subsection{Related Notions: Combinatorial and Discrete Curvature}

A notion of combinatorial curvature was introduced by Gromov in the 1980s to study hyperbolic groups, which has since been widely adapted by both discretizing classical smooth notions and adapting notions for discrete settings \citep{gromov1987hyperbolic,ishida1990pseudo}. Notably, various versions of curvature have been developed for graphs, Markov chains, piecewise flat shapes, and polyhedral settings---including polyhedra, noncompact polyhedra, and polyhedral surfaces \citep{cheeger2015lower,cheeger1984curvature,ollivier2009ricci,banchoff1970critical}.  All of these ideas fall within the broad and active research direction of discrete curvature and discrete geometry \citep{najman2014discrete,najman2017modern}.  These notions may also be studied for possible alternative curvatures for the tropical projective torus, however, these are entirely different approaches which we do not take in this work.  

A relative disadvantage of many alternative curvature notions and their discrete counterparts is that they are often defined using quite complicated and long formulae, which makes it difficult to assess whether these notions are intrinsic or not.  This is in contrast to more geometric definitions---such as Alexandrov curvature---and why these may be more preferable to study. 

We note here, however, that is indeed possible to study alternative discrete notions of curvature using existing work on the tropical optimal transport problem, which has been set up and used to derive tropical Wasserstein distances: solving the optimal transport problem with respect to the tropical metric and using Dirac masses will the {\em Ollivier--Ricci curvature} of the tropical projective torus \citep{OLLIVIER2007643,Lee2021}.

%%%%%%%%%%%%%%%%%%%%%%%%%%%%%%%%%%%%%%%%%%%%%%%%%%%

\section{Negative and Positive Alexandrov Curvature in $\R^3/\R\one$:\\ A Sampling Procedure and a Theorem}
\label{sec:alexandrov_plusminus}

Let's begin by computing Alexandrov curvature in $(\R^3/\R\one, \dtr)$ for some example triangles.

\begin{remark}
In all of the examples to follow where we give explicit computations of specific triangles, the legend for the corresponding figures of the comparison distances is as follows: the $x$-axis is the value of the parameter $t$; the $y$-axis is the value of the respective tropical and Euclidean distances as a function of $t$; the green line is always the tropical distance; and the blue line is always the Euclidean distance. 
\end{remark}

\begin{figure}[ht]
\centering
\subfigure[]{
\begin{tikzpicture}
\coordinate (a) at (1,3);
\coordinate (b) at (0,0);
\coordinate (c) at (3,2);
\coordinate (k1) at (1,2);
\coordinate (k2) at (1,1);
\coordinate (k3) at (2,2);
\foreach \coord in {a,b,c}
	\draw[fill=black] (\coord) circle[radius=1pt];
\foreach \coord in {a,c}
	\node[label={[label distance=-2pt]90:$\coord$}]at (\coord) {};
\foreach \coord in {b}
\node[label={[label distance=-20pt]:$\coord$}] at (\coord) {};
\draw (a)--(k2);
\draw (k3)--(b);
\draw (c)--(k1);
\end{tikzpicture}}
\hskip3cm
\subfigure[]{
\begin{tikzpicture}
\coordinate (a) at (1.5,2.598);
\coordinate (b) at (0,0);
\coordinate (c) at (3,0);
\foreach \coord in {a}
	\node[label={[label distance=-15pt]270:$\coord'$}]at (\coord) {};
	\foreach \coord in {b}
	\node[label={[label distance=1pt]270:$\coord'$}]at (\coord) {};
\foreach \coord in {c}
	\node[label={[label distance=1pt]270:$\coord'$}]at (\coord) {};
\draw (a) -- (b) -- (c) -- (a);
\end{tikzpicture}}
\caption{(a) Example of a skinny tropical triangle; (b) Corresponding Euclidean comparison triangle.}
\label{fig:skinny}
\end{figure}
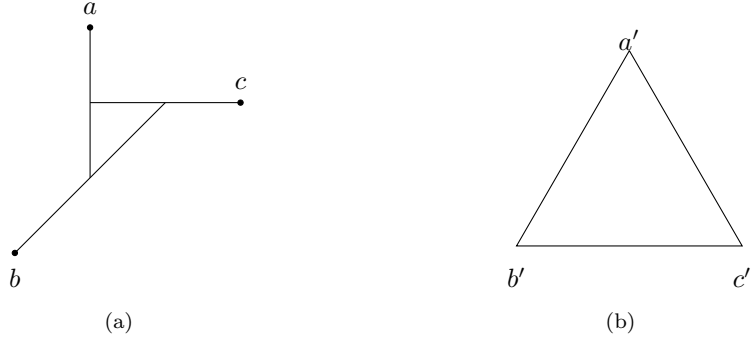

\begin{example}
\label{ex:eqskinny}
We compute the Alexandrov curvature of the tropical triangle depicted in Figure \ref{fig:skinny}(a).  The vertices are $a=(1,3),\, b=(0,0),$ and $c=(3,2)$ in $\R^3/\R\one$ (recall the notational convention from Remark~\ref{rem:notation}).

The lengths of the three edges of the tropical triangle are
\begin{align*}
\dtr(a,b) & = \dtr([(0,1,3)
,\, [(0,0,0)]) = 
%& = \max((0-0),\, (1-0),\, (3-0)) - \min((0-0),\, (1-0),\, (3-0))\\
\max(0,1,3) - \min(0,1,3) = 3-0 = 3\\
\dtr(b,c) & = \dtr([(0,0,0)],\, [(0,3,2)])
%= \max((0-0),\, (0-3),\, (0-2)) - \min((0-0),\, (0-3),\, (0-2))\\
= \max(0,-3,-2) - \min(0,-3,-2) = 3\\
\dtr(a,c) & = \dtr([(0,1,3)],\, [(0,3,2)]) = \max(0,-2,1) - \min(0,-2,1) = 3;
\end{align*}
i.e., it is a tropical equilateral triangle.

The comparison triangle in Euclidean space is depicted in Figure \ref{fig:skinny}(b) has vertices at $a' = (\frac{3}{2}, \frac{3\sqrt{3}}{2})$, $b' = (0,0)$, and $c'=(0,3)$; it is an equilateral triangle with $d_e(a', b') = d_e(b', c') = d_e(a', c') = 3$.

\paragraph{Computing tropical and Euclidean distances between the vertex $a$ and the $bc$ edge.}
Notice that the $bc$ edge connecting the vertex $b$ to the vertex $c$ is the tropical line segment computed in Example \ref{ex:tls_1}.  We now compute the tropical distance between the vertex $a$ and the $bc$ edge (\ref{eq:tls_ex1}) as
$$
\dtr((1,3),\, \gamma_{bc}(t)) = \begin{cases}
\dtr((1,3),\, (t,t)), & 0 \leq t < 2;\\
\dtr((1,3),\, (t,2)), & 2 \leq t \leq 3.
\end{cases}
$$
When $0 \leq t < 2$,
\begin{align*}
\dtr((1,3),\, (t,t)) & = \max(0, 1-t, 3-t) - \min(0, 1-t, 3-t)\\
& = \left\{\begin{array}{lll}
(3-t) - 0 & = 3-t, & \mbox{~~}0 \leq t < 1\\
(3-t) - (1-t) & = 2, & \mbox{~~} 1 \leq t < 2.
\end{array}
\right.
\end{align*}
When $2 \leq t \leq 3$,
\begin{align*}
\dtr((1,3),\, (t,2)) & = \max(0, 1-t, 1) - \min(0, 1-t, 1)\\
& = 1-(1-t) = t.
\end{align*}
This gives
\begin{equation}
\label{eq:skinny_abc}
\dtr((1,3),\, \gamma_{bc}(t)) = \begin{cases}
3-t, & 0 \leq t < 1,\\
2, & 1 \leq t < 2,\\
t, & 2 \leq t \leq 3.
\end{cases}
\end{equation}

\begin{figure}[h!]
\centering
\includegraphics[scale=0.35]{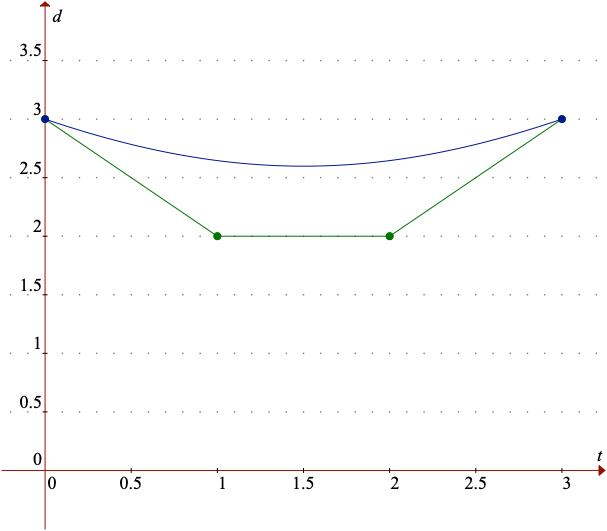}
\caption{Tropical distance function vs. Euclidean distance function from Example~\ref{ex:eqskinny}}
\label{fig:type1eq}
\end{figure}

To study Alexandrov curvature, we need to compare the tropical distance from $a$ to any point on the tropical line segment $\gamma_{bc}(t)$ to the Euclidean distance from $a'$ to any point on the $b'c'$ edge.  We measure the position $t \in [0, 3]$ from $a'$ and compute 
$$
h^2:= d_e^2(a',\gamma_{b'c'}(t)) = \bigg(\frac{3\sqrt{3}}{2}\bigg)^2 + \bigg(\frac{3}{2}-t\bigg)^2 = \frac{27}{4} + \bigg(\frac{3}{2}-t\bigg)^2;
$$
notice that the minimum is achieved at $\frac{3}{2}$.

Comparing $h^2$ to (\ref{eq:skinny_abc}), we find $\dtr^2((1,3),\, \gamma_{bc}(t)) \leq h^2$ and
$$
\dtr(a, \gamma_{bc}(t)) \leq d_e(a', \gamma_{b'c'}(t)).
$$
%\begin{align*}
%(3-t)^2 \leq \frac{27}{4} + \bigg(\frac{3}{2}-t\bigg)^2 & \Longleftrightarrow \\
%2^2 \leq \frac{27}{4} + \bigg(\frac{3}{2}-t\bigg)^2 & \Longleftrightarrow\\
%t^2 \leq \frac{27}{4} + \bigg(\frac{3}{2}-t\bigg)^2 & \Longleftrightarrow
%\end{align*}
Figure \ref{fig:type1eq} displays the curves of the tropical and Euclidean distances; we see that the Euclidean distance is always greater than the tropical distance.  We therefore conclude that Alexandrov curvature of the vertex $a$ to the $bc$ edge of this triangle assessed using the tropical distance from the vertex $a$ to the tropical line segment $\gamma_{bc}(t)$ is negative.

Given that both triangles are equilateral under their respective metrics, a verbatim calculation performed on a relabeling of vertices and edges yields the same conclusion.  Thus, the Alexandrov curvature of $(\R^3/\R\one, \dtr)$ computed with respect to the triangle $(a,b,c)$ is negative (nonpositive): this tropical triangle is skinnier than its Euclidean comparison triangle.
\end{example}

\begin{figure}[ht]
\centering
\subfigure[]{
\begin{tikzpicture}
\coordinate (a) at (0,2);
\coordinate (b) at (1,0);
\coordinate (c) at (3,3);
\coordinate (k1) at (0,0);
\coordinate (k2) at (1,3);
\coordinate (k3) at (3,2);
\foreach \coord in {a,b,c}
	\draw[fill=black] (\coord) circle[radius=1pt];
\foreach \coord in {a,c}
	\node[label={[label distance=-2pt]90:$\coord$}]at (\coord) {};
\foreach \coord in {b}
\node[label={[label distance=-20pt]:$\coord$}] at (\coord) {};
\draw (a)--(k1);
\draw (a)--(k2);
\draw (k2)--(c);
\draw (k1)--(b);
\draw (k3)--(b);
\draw (k3)--(c);
\end{tikzpicture}}
\hskip3cm
\subfigure[]{
\begin{tikzpicture}
\coordinate (a) at (1.5,2.598);
\coordinate (b) at (0,0);
\coordinate (c) at (3,0);
\foreach \coord in {a}
	\node[label={[label distance=-15pt]270:$\coord'$}]at (\coord) {};
	\foreach \coord in {b}
	\node[label={[label distance=1pt]270:$\coord'$}]at (\coord) {};
\foreach \coord in {c}
	\node[label={[label distance=1pt]270:$\coord'$}]at (\coord) {};
\draw (a) -- (b) -- (c) -- (a);
\end{tikzpicture}}
\caption{(a) Example of a fat tropical triangle; (b) Corresponding Euclidean comparison triangle.}
\label{fig:fat}
\end{figure}
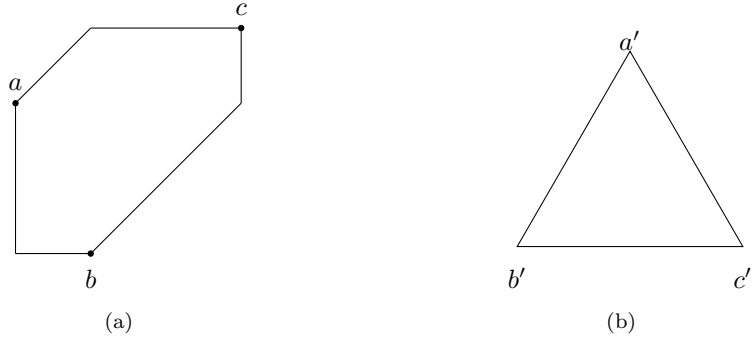

\begin{example}
\label{ex:eqfat}
We compute the Alexandrov curvature of the tropical triangle depicted in Figure \ref{fig:fat}(a). The vertices are $a = (0,2),\, b = (1,0)$, and $c = (3,3)$ in $\R^3/\R\one$. The lengths of the three edges of this tropical triangle are
\begin{align*}
\dtr(a,b) & = \dtr([(0,0,2)],\, [(0,1,0)]) = \max(0, -1, 2) - \min(0, -1, 2) = 3\\
\dtr(b,c) & = \dtr([(0,1,0)],\, [(0,3,3)]) = \max(0, -2, -3) - \min(0, -2, -3) = 3\\
\dtr(a,c) & = \dtr([(0,0,2)],\, [(0,3,3)]) = \max(0, -3, -1) - \min(0, -3, -1) = 3;
\end{align*}
i.e., it is a tropical equilateral triangle.

The Euclidean comparison triangle is depicted in Figure \ref{fig:fat}(b) with vertices at $a' = (\frac{3}{2}, \frac{3\sqrt{3}}{2}),\, b' = (0,0)$, and $c' = (0,3)$, as above in Example \ref{ex:eqskinny}.

\paragraph{Computing tropical and Euclidean distances between the vertex $a$ and the $bc$ edge.}

The $bc$ edge is the tropical line segment connecting the vertex $b$ to the vertex $c$:
$$
\gamma_{bc}(t) = t \odot [(0,1,0)] \oplus [(0,3,3)] = (\min(t,0),\, \min(t+1,3),\, \min(t,3)).
$$
For $t\in [0,3]$, we have
$$
\gamma_{bc}(t) = \begin{cases}
[(0, t+1, t)], & 0 \leq t < 2;\\
[(0, 3, t)], & 2 \leq t \leq 3.
\end{cases}
$$

The tropical distance between the vertex $a$ and the $bc$ edge is
$$
\dtr([(0,0,2)],\, \gamma_{vw}(t)) = \begin{cases}
\dtr([(0,0,2)],\, [(0, t+1, t)]), & 0 \leq t < 2;\\
\dtr([(0,0,2)],\, [(0, 3, t)]), & 2 \leq t \leq 3.
\end{cases}
$$
When $0 \leq t < 2$,
\begin{align*}
\dtr([(0,0,2)],\, [(0, t+1, t)]) & = \max(0,-t-1, 2-t) - \min(0,-t-1, 2-t)\\
& = 2-t - (-t-1) = 3.
\end{align*}
When $2 \leq t \leq 3$,
\begin{align*}
\dtr([(0,0,2)],\, [(0, 3, t)]) & = \max(0,-3, 2-t) - \min(0,-3, 2-t)\\
& = 0 - (-3) = 3.
\end{align*}
So $\dtr([(0,0,2)],\, \gamma_{vw}(t)) = 3$ for all $t \in [0,3]$.

\begin{figure}[h!]
\centering
\includegraphics[scale=0.35]{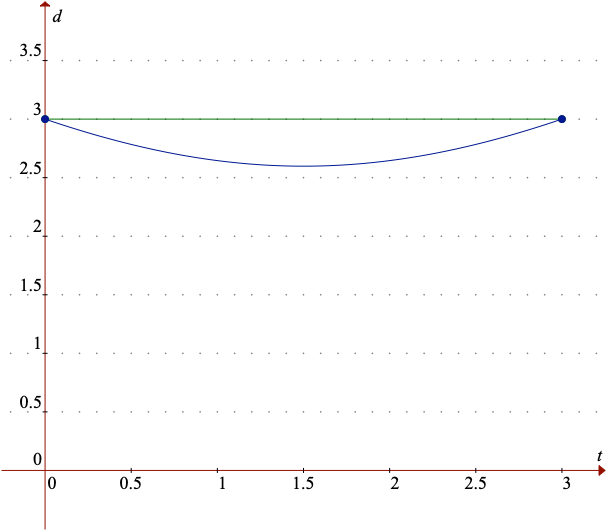}
\caption{Tropical distance function vs. Euclidean distance function from Example~\ref{ex:eqfat}}
\label{fig:type5eq}
\end{figure}

As in Example \ref{ex:eqskinny}, we compare the tropical distance from $a$ to any point on the tropical line segment $\gamma_{bc}(t)$ to the Euclidean distance from $a'$ to any point on the $b'c'$ edge.  We measure the position $t \in [0,3]$ from $u_0$ and compute
$$
h^2:= d_e^2(a',t) = \bigg(\frac{3\sqrt{3}}{2}\bigg)^2 + \bigg(\frac{3}{2}-t\bigg)^2 = \frac{27}{4} + \bigg(\frac{3}{2}-t\bigg)^2.
$$
Since $\dtr^2((0,0,2),\, \gamma_{ac}(t)) = 9 \geq h^2$ for all $t$, we conclude that the Alexandrov curvature assessed using the tropical distance from the vertex $a$ to the tropical line segment $\gamma_{bc}(t)$ is positive.

Again, given that both triangles are equilateral under their respective metrics, we may similarly to Example \ref{ex:eqskinny} conclude that the Alexandrov curvature of $(\R^3/\R\one, \dtr)$ computed with respect to the triangle $(a,b,c)$ is positive (nonnegative): this tropical triangle is fatter than its Euclidean comparison triangle.
\end{example}

Since the definition of Alexandrov curvature relies on the Euclidean comparison triangle, here, we establish a first general result from Euclidean trigonometry that we will repeatedly use throughout this paper.
% used in Examples \ref{ex:eqskinny} and \ref{ex:eqfat} above 

\begin{figure}
\centering
\begin{tikzpicture}
\coordinate (a') at (0,0);
\coordinate (b') at (2,3);
\coordinate (c') at (5,0);
\coordinate (x) at (3,0);
\foreach \coord in {a,b,c}
	\draw[fill=black] (\coord') circle[radius=1pt];
\foreach \coord in {a,b,c}
	\node[label={[label distance=-2pt]90:$\coord'$}]at (\coord') {};
\draw (a')--(b')--(c')--(a');
\draw (b')--(x);
\end{tikzpicture}
\put(-50,50){{$A$}}
\put(-125,50){{$C$}}
\put(-90,-7){{$B$}}
\put(-70,35){{$h$}}
\put(-110,7){{$t$}}
\put(-140,7){{$\theta$}}
\put(-64,-5){{$x$}}
\caption{Euclidean Vertex-to-Edge Distance.}
\label{fig:euclidist}
\end{figure}
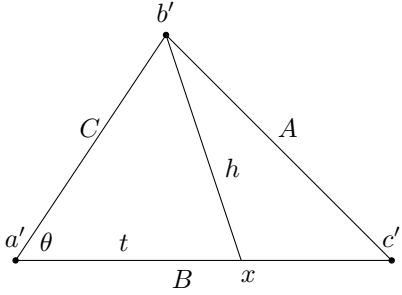

\begin{lemma}\label{lem:euclidist}
Consider a Euclidean triangle in the plane with vertices $a', b', c'$ and opposite side lengths $A,B,C$ respectively (Figure \ref{fig:euclidist}).  If $x$ is the point on the side $a'c'$ such that $\|a'-x\| = t$, then the distance $h$ from $b'$ to $x$ is given by
\begin{equation*}
\label{eq:euclidist}
h^2 = t^2 + \frac{A^2-B^2-C^2}{B}t + C^2.
\end{equation*}
\end{lemma}
\begin{proof}
We obtain the result by applying the Law of Cosines twice.  Let $\theta := \angle b'a'c'$; then $A^2 = B^2 + C^2 - 2BC\cos\theta$, from which we obtain
\begin{equation}
\label{eq:cosalpha}
-2C \cos \theta = \frac{A^2 - B^2 - C^2}{B}.
\end{equation}
Similarly,
\begin{align*}
h^2 & = t^2 + C^2 - 2tC \cos \theta\\
& = t^2 + C^2 + \frac{A^2 - B^2 - C^2}{B}t,
\end{align*}
by substituting (\ref{eq:cosalpha}), as desired.
\end{proof}

Examples \ref{ex:eqskinny} and \ref{ex:eqfat}, together with Corollary \ref{cor:cat}, show that characterizing the curvature behavior of $(\R^n/\R\one, \dtr)$ is not straightforward.  They motivate the driving question of this paper: what is the general behavior of Alexandrov curvature under the tropical metric?

Before approaching this question, we make a quick detour to discuss combinatorial types of tropical triangles.

\subsection{Combinatorial Types of Tropical Triangles in the Plane}

Triangles in tropical geometry are interesting objects in their own right that have been previously studied, for example in \cite{ansola2009note}; in our work, they are fundamental.  In this paper, our triangles have edges given by tropical line segments as defined in Definition \ref{def:tls}: from Lemma \ref{lemma:tls_types}, since there are only three possible types of tropical line segment, this restricts the number of combinatorial types of tropical triangles as a combination of the types of tropical line segments. Mike Develin and Bernd Sturmfels show that in the plane, there are five combinatorial types of tropical triangles \citep{develin2004tropical}, illustrated in Figure \ref{fig:5types}. They proved in general that the combinatorial types of tropical complexes generated by a set of $r$ vertices in $\mathbb{R}^n/\mathbb{R}\one$ are in natural bijection with the regular polyhedral subdivisions of the product of simplices $\Delta_{n-1} \times \Delta_{r-1}$. See also connections to tropical convexity, and point configurations in a Bruhat--Tits building by Michael Joswig, Bernd Sturmfels, and Josephine Yu; and Dustin Cartwright, Bernd Sturmfels and co-authors \citep{joswig2007josephine,cartwright2011mustafin}. %Each is defined by a set of inequalities on the coordinates of the vertices.

\begin{figure}
\subfigure[]{
\centering
\resizebox{30mm}{!}{
\begin{tikzpicture}
\coordinate (a) at (0,0);
\coordinate (b) at (2,3);
\coordinate (c) at (4,1);
\coordinate (k1) at (2,2);
\coordinate (k2) at (2,1);
\coordinate (k3) at (1,1);
\foreach \coord in {a,b,c}
	\draw[fill=black] (\coord) circle[radius=1pt];
\foreach \coord in {a,b,c}
	\node[label={[label distance=-2pt]90:$\coord$}]at (\coord) {};
\draw (a)--(k1);
\draw (b)--(k2);
\draw (c)--(k3);
\end{tikzpicture}}}
\hskip5mm
\subfigure[]{
\centering
\resizebox{30mm}{!}{
\begin{tikzpicture}
\coordinate (a) at (0,0);
\coordinate (b) at (3,2);
\coordinate (c) at (4,1);
\coordinate (k1) at (2,2);
\coordinate (k2) at (3,1);
\coordinate (k3) at (1,1);
\foreach \coord in {a,b,c}
	\draw[fill=black] (\coord) circle[radius=1pt];
\foreach \coord in {a,b,c}
	\node[label={[label distance=-2pt]90:$\coord$}]at (\coord) {};
\draw (a)--(k1);
\draw (b)--(k2);
\draw (c)--(k3);
\draw (k1)--(b);
\end{tikzpicture}}}
\hskip5mm
\subfigure[]{
\resizebox{30mm}{!}{
\begin{tikzpicture}
\coordinate (a) at (0,0);
\coordinate (b) at (3,1);
\coordinate (c) at (5,2);
\coordinate (k1) at (2,2);
\coordinate (k2) at (4,2);
\coordinate (k3) at (1,1);
\foreach \coord in {a,b,c}
	\draw[fill=black] (\coord) circle[radius=1pt];
\foreach \coord in {a,b,c}
	\node[label={[label distance=-2pt]90:$\coord$}]at (\coord) {};
\draw (a)--(k1);
\draw (b)--(k2);
\draw (c)--(k1);
\draw (k3)--(b);
\end{tikzpicture}}}
\hskip5mm
\subfigure[]{
\resizebox{30mm}{!}{
\begin{tikzpicture}
\coordinate (a) at (0,1);
\coordinate (b) at (1,0);
\coordinate (c) at (4,2);
\coordinate (k1) at (0,0);
\coordinate (k2) at (1,2);
\coordinate (k3) at (3,2);
\foreach \coord in {a,b,c}
	\draw[fill=black] (\coord) circle[radius=1pt];
\foreach \coord in {a,b,c}
	\node[label={[label distance=-2pt]90:$\coord$}]at (\coord) {};
\draw (a)--(k1)--(b)--(k3);
\draw (a)--(k2)--(c);
\end{tikzpicture}}}
\hskip16mm
\subfigure[]{
\resizebox{20mm}{!}{
\begin{tikzpicture}
\coordinate (a) at (0,1);
\coordinate (b) at (1,0);
\coordinate (c) at (2,2);
\coordinate (k1) at (0,0);
\coordinate (k2) at (1,2);
\coordinate (k3) at (2,1);
\foreach \coord in {a,b,c}
	\draw[fill=black] (\coord) circle[radius=1pt];
\foreach \coord in {a,b,c}
	\node[label={[label distance=-2pt]90:$\coord$}]at (\coord) {};
\draw (a)--(k1)--(b)--(k3)--(c)--(k2)--(a);
\end{tikzpicture}}}
\caption{Combinatorial Types of Tropical Triangles in the Plane: (a) T1; (b) T2; (c) T3; (d) T4; (e) T5.}
\label{fig:5types}
\end{figure}
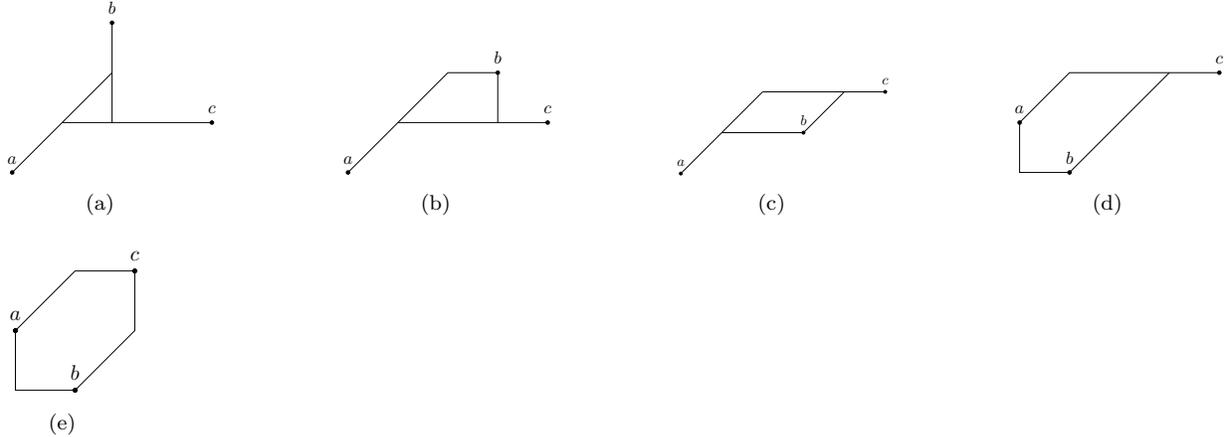

\subsection{Regions of Skinniness and Fatness: Sampling Tropical Triangles}

From Examples \ref{ex:eqskinny} and \ref{ex:eqfat}, a first natural assumption would be that type T1 triangles are always skinny, while type T5 triangles are always fat.  It turns out that the story is not so simple.

If we think of all tropical triangles as parametrized by their coordinates
$$
(a_1,a_2,b_1,b_2,c_1,c_2) \in \R^6
$$
(including degenerate collinear cases), the conditions for having a positive or negative Alexandrov curvature are given by a set of linear and quadratic inequalities that depend on the specific configuration of the triangle. In particular, these regions are \emph{semialgebraic sets}. We would like for these to have nonempty interior and in particular, that the set of triangles with a specific curvature is not of measure zero.  We would like to wiggle the coordinates in Examples \ref{ex:eqskinny} (type T1) and \ref{ex:eqfat} (type T5) and preserve the curvature; the plots for these examples suggest that this is possible, because there is a clear gap between the curves of the tropical and Euclidean distances.  From the examples themselves, this is not obvious since these are equilateral and there is only one Euclidean distance function to which we compare the tropical distance.  What happens when we have scalene triangles?

%It turns out that these assumptions are false, which is not straightforward.  Given the existence of different types of tropical triangles, this makes the story quite complicated.

%In Example \ref{ex:T1_skinny}, we actually find a configuration of vertices of type T1 that show that the curvature is still preserved, so we do still have negative curvature even with a generic scalene triangle. However, Example~\ref{ex:T5_222} illustrates that equality between Euclidean and tropical can be achieved in the middle of a segment, and wiggling in the wrong direction could lead to an inequality being violated, so there is a complex and sensitive behavior, which we investigated numerically.

\begin{example}
\label{ex:T1_skinny}
Take the vertex set in $\R^3/\R\one$:
$$
a = (0,0), \qquad b = (2,4), \qquad c = (5,1).
$$
This is a scalene triangle of type T1 illustrated in Figure \ref{fig:type1scalene}(d); it has edge lengths $\dtr(a,b) = 4$, $\dtr(a,c) = 5$, and $\dtr(b,c) = 6$.  We compute all the tropical distances between each vertex to its opposite edge and compare them to their corresponding Euclidean distances.

%The tropical distance from the vertex $b$ to the $ac$ edge is
%$$
%\dtr(b,\, \gamma_{ac}(t)) = \begin{cases}
%4-t, & 0 \leq t \leq 1;\\
%3, & 1 \leq t \leq 2;\\
%t+1, & 2 \leq t \leq 5.
%\end{cases}
%$$
%The squared Euclidean distance between the vertex $b'$ and the $a'c'$ edge on the comparison triangle is
%$$
%d_e(b',\, \gamma_{a'c'}(t)) = t^2 - t + 16.
%$$

%Similarly, the tropical distance from the vertex $c$ to the $ab$ edge is
%$$
%\dtr(c,\, \gamma_{ab}(t)) = \begin{cases}
%5-t, & 0 \leq t \leq 1;\\
%4, & 1 \leq t \leq 2;\\
%t+2, & 2 \leq t \leq 4,
%\end{cases}
%$$
%while the squared Euclidean distance from the vertex $c'$ to the $a'b'$ edge on the comparison triangle is
%$$
%d_e(c',\, \gamma_{a'b'}(t)) = t^2 - \frac{5}{4}t + 25.
%$$

%Finally, the tropical distance from the vertex $a$ to the $bc$ edge is
%$$
%\dtr(a,\, \gamma_{bc}(t)) = \begin{cases}
%4-t, & 0 \leq t \leq 2;\\
%2, & 2 \leq t \leq 3;\\
%t-1, & 3 \leq t \leq 6,
%\end{cases}
%$$
%and the corresponding squared Euclidean distance from the vertex $a'$ to the $b'c'$ edge is
%$$
%d_e(a',\, \gamma_{b'c'}(t)) = t^2 - \frac{27}{6}t + 16.
%$$
We see from the comparison plots of the tropical versus Euclidean distances in Figure \ref{fig:type1scalene} that this scalene tropical triangle is skinny, i.e., it has negative (nonpositive) Alexandrov curvature, since the curves for the tropical distances always lie below the curves for the Euclidean distances.
\end{example}

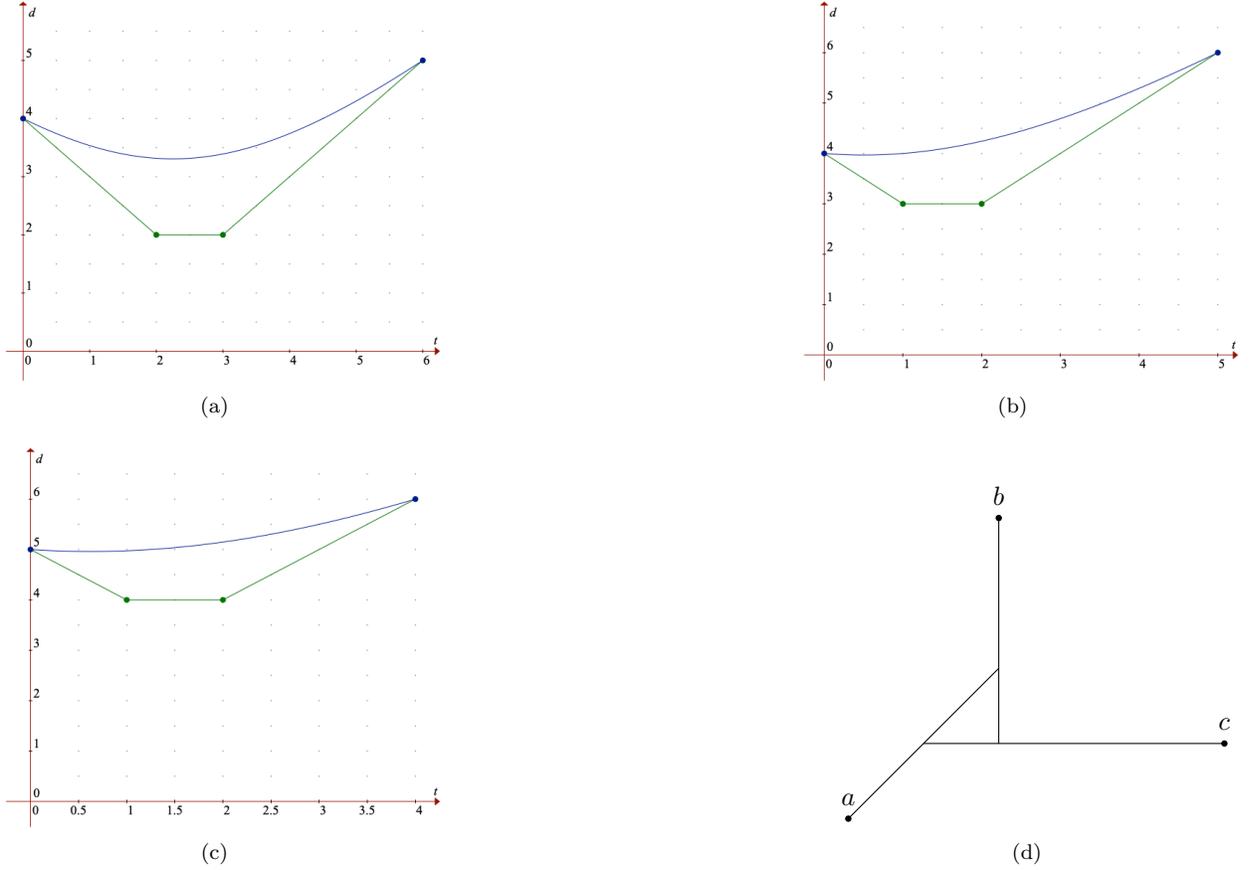
\begin{figure}[h!]
\subfigure[]{
\centering
\includegraphics[scale=0.27]{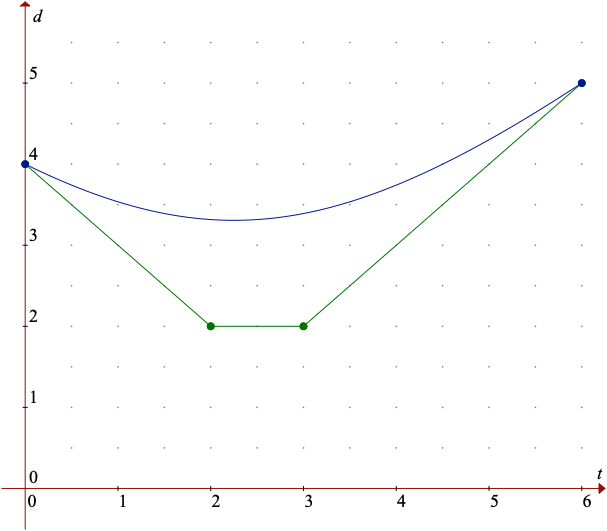}}
\hskip3mm
\subfigure[]{
\centering
\includegraphics[scale=0.27]{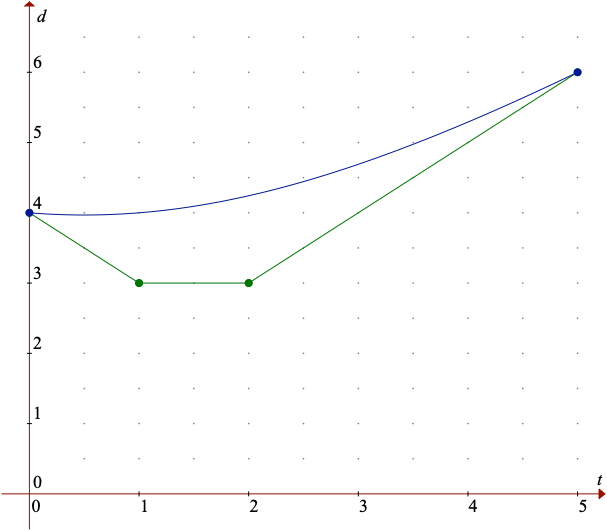}}
\subfigure[]{
\centering
\includegraphics[scale=0.27]{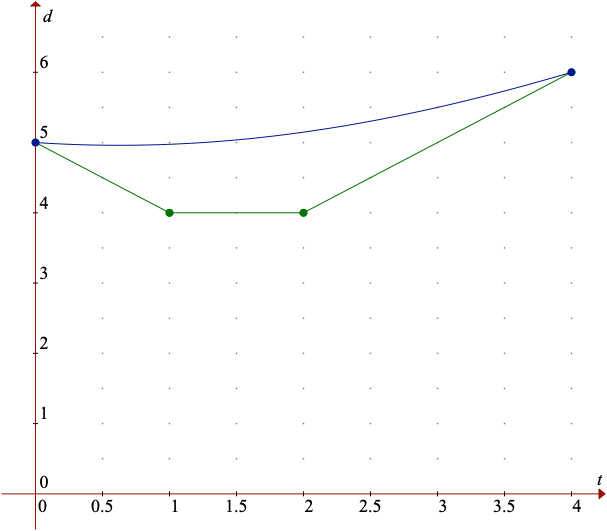}}
\hskip50mm
\subfigure[]{
\centering
%\resizebox{30mm}{!}{
\begin{tikzpicture}
\coordinate (a) at (0,0);
\coordinate (b) at (2,4);
\coordinate (c) at (5,1);
\coordinate (k1) at (2,2);
\coordinate (k2) at (2,1);
\coordinate (k3) at (1,1);
\foreach \coord in {a,b,c}
	\draw[fill=black] (\coord) circle[radius=1pt];
\foreach \coord in {a,b,c}
	\node[label={[label distance=-2pt]90:$\coord$}]at (\coord) {};
\draw (a)--(k1);
\draw (b)--(k2);
\draw (c)--(k3);
\end{tikzpicture}}
\caption{Tropical distance function vs. Euclidean distance function from Example~\ref{ex:T1_skinny}: (a) vertex $a$ to side $bc$; (b) vertex $b$ to side $ac$; (c) vertex $c$ to side $ab$; (d) Type 1 tropical triangle studied in Example \ref{ex:T1_skinny}.}
\label{fig:type1scalene}
\end{figure}

%Example~\ref{ex:T1_skinny} shows that in general we need to check all three sides to compare distinct tropical distance functions to distinct Euclidean functions.

\begin{example}
\label{ex:T5_222}
Consider the vertex set
$$
a = (0,4), \qquad b = (3,0), \qquad c = (5,6).
$$
This is a scalene triangle of type T5, illustrated in Figure \ref{fig:type5scalene}(d); it has edge lengths $\dtr(a,b) = 7$, $\dtr(a,c) = 5$, and $\dtr(b,c) = 6$.  We compute all the tropical distances between each vertex to its opposite edge and compare them to their corresponding Euclidean distances.

%The tropical distance from the vertex $b$ to the $ac$ edge is
%$$
%\dtr((b_1, b_2),\, \gamma_{ac}(t)) = \begin{cases}
%7, & 0 \leq t \leq 2;\\
%9-t, & 2 \leq t \leq 3;\\
%6, & 3 \leq t \leq 5.
%\end{cases}
%$$
%The squared Euclidean distance between the vertex $b'$ and the $a'c'$ edge on the comparison triangle is
%$$
%d_e^2(b',\, \gamma_{a'c'}(t)) = t^2 - \frac{38}{5}t + 49.
%$$

%Similarly, the tropical distance from the vertex $c$ to the $ab$ edge is
%$$
%\dtr((c_1, c_2),\, \gamma_{ab}(t)) = \begin{cases}
%5, & 0 \leq t \leq 3;\\
%t+2, & 3 \leq t \leq 4;\\
%6, & 4 \leq t \leq 7,
%\end{cases}
%$$
%while the squared Euclidean distance from the vertex $c'$ to the $a'b'$ edge on the comparison triangle is
%$$
%d_e^2(c',\, \gamma_{a'b'}(t)) = t^2 - \frac{38}{7}t + 25.
%$$

%Finally, the tropical distance from the vertex $a$ to the $bc$ edge is
%$$
%\dtr((a_1, a_2),\, \gamma_{bc}(t)) = \begin{cases}
%7, & 0 \leq t \leq 2;\\
%9-t, & 2 \leq t \leq 4;\\
%5, & 4 \leq t \leq 6,
%\end{cases}
%$$
%and the corresponding squared Euclidean distance from the vertex $a'$ to the $b'c'$ edge is
%$$
%d_e^2(a',\, \gamma_{b'c'}(t)) = t^2 - 10t + 49.
%$$
From the comparison plots of the tropical versus Euclidean distances shown in Figure \ref{fig:type5scalene}, we see that this scalene tropical triangle is fat, i.e., it has positive (nonnegative) Alexandrov curvature, since the curves for the tropical distances always lie above the curves for the Euclidean distances.
\end{example}

\begin{figure}[h!]
\subfigure[]{
\centering
\includegraphics[scale=0.27]{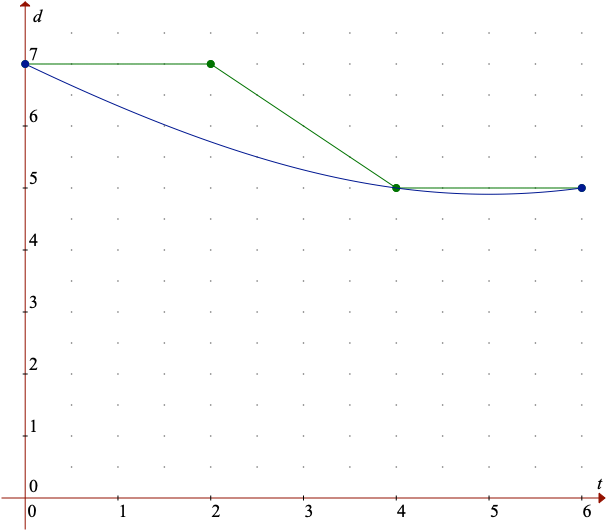}}
\hskip3mm
\subfigure[]{
\centering
\includegraphics[scale=0.27]{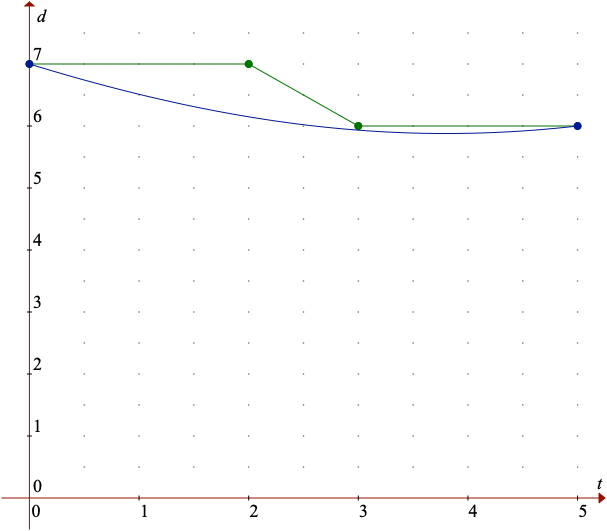}}
\subfigure[]{
\centering
\includegraphics[scale=0.27]{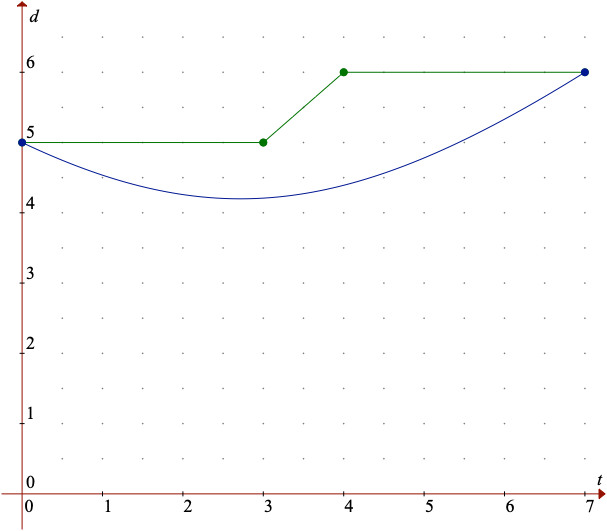}}
\hskip50mm
\subfigure[]{
\begin{tikzpicture}
\coordinate (a) at (0,4);
\coordinate (b) at (3,0);
\coordinate (c) at (5,6);
\coordinate (k1) at (0,0);
\coordinate (k2) at (2,6);
\coordinate (k3) at (5,2);
\foreach \coord in {a,b,c}
	\draw[fill=black] (\coord) circle[radius=1pt];
\foreach \coord in {a,b,c}
	\node[label={[label distance=-2pt]90:$\coord$}]at (\coord) {};
\draw (a)--(k1)--(b)--(k3)--(c)--(k2)--(a);
\end{tikzpicture}}
\caption{Tropical distance function vs. Euclidean distance function from Example~\ref{ex:T5_222} (a) vertex $a$ to side $bc$; (b) vertex $b$ to side $ac$; (c) vertex $c$ to side $ab$; (d) Type 5 tropical triangle studied in Example \ref{ex:T5_222}.}
\label{fig:type5scalene}
\end{figure}

The example just studied above is particularly interesting, because we see in Figure \ref{fig:type5scalene}(a), the two distance curves touch; the curves in Figure \ref{fig:type5scalene}(b) come very close.  This indicates that the directions of the inequalities may not be preserved throughout the whole line segment and Alexandrov curvature may turn out to be undefined for some triangles.  The implications of undefined Alexandrov curvature will be addressed further on in the paper in Section~\ref{sec:undefined}.  Undefined curvature can happen with both type T1 and T5 triangles, but with very different occurrences: it turns out it is very easy to find T1 triangles that are skinny but there do exist some that are undefined (see Example \ref{ex:T1_counterexample} immediately below); conversely, it is very difficult to find T5 triangles that are fat, they are almost always undefined.  We conduct numerical experiments to understand how often T1 triangles are skinny and undefined, and similarly, how often T5 triangles are fat and undefined.

\begin{example}
\label{ex:T1_counterexample}
Consider the tropical triangle in $\R^3/\R\one$ defined by the vertices
$$
a = (0,0), \qquad b = (448, 449), \qquad c = (452, 256). 
$$

The tropical distance from the vertex $b$ to the $ac$ edge is
$$
\dtr((b_1, b_2),\, \gamma_{ac}(t)) = \begin{cases}
449-t, & 0 \leq t \leq 256;\\
193, & 256 \leq t \leq 448;\\
t-255, & 448 \leq t \leq 452,
\end{cases}
$$
while the squared corresponding Euclidean distances on the comparison triangle from vertex $b'$ to the $a'c'$ edge is
$$
d_e((b'),\, \gamma_{a'c'}(t)) = t^2 - \frac{91774}{113}t + 201601.
$$
Here, the tropical distance begins below the corresponding Euclidean distance (i.e., the triangle starts out skinny), but then coincide, and the tropical distance then continues {\em greater} than the Euclidean distance (i.e., it becomes fat), as shown in Figure \ref{fig:type1counterex}, before eventually dipping below the curve of the Euclidean distance again (i.e., it becomes skinny again).  So, we have a change of sign in curvature over an interval and hence, Alexandrov curvature is not well-defined for this triangle.  Note that the change of sign only occurs over a very small interval as shown in Figure \ref{fig:type1counterex}(a) and magnified in Figure \ref{fig:type1counterex}(b).

\begin{figure}[h!]
\subfigure[]{
\centering
\includegraphics[scale=0.27]{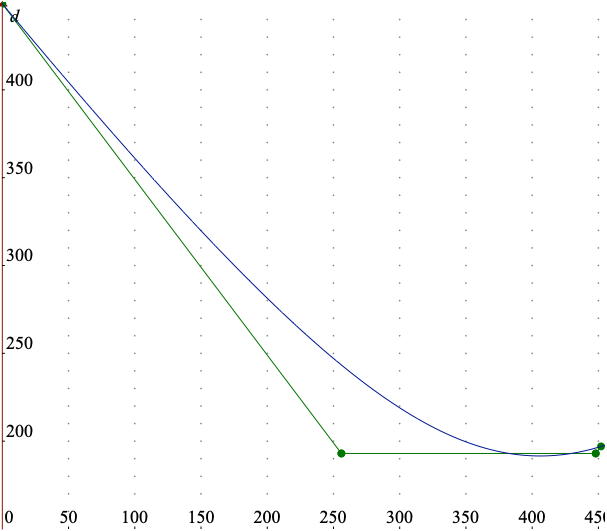}}
\hskip35mm
\subfigure[]{
\centering
\includegraphics[scale=0.27]{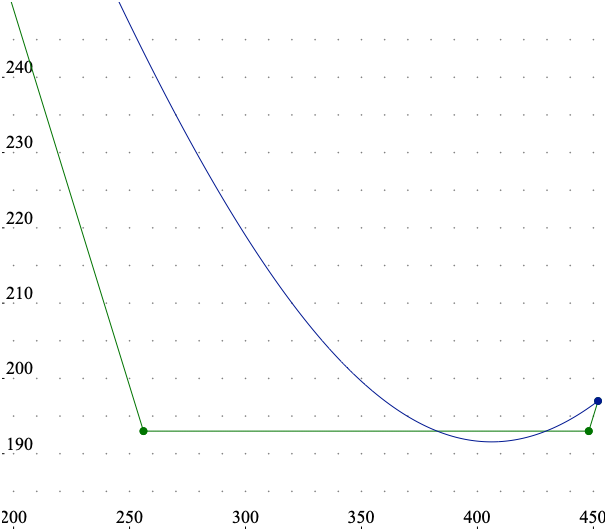}}
\caption{Tropical distance function vs. Euclidean distance function from Example~\ref{ex:T1_counterexample} for side $ac$: (a) Comparison of the two distance functions; (b) Magnified portion to see where the distance functions cross and the curvature of this triangle becomes undefined.}
\label{fig:type1counterex}
\end{figure}

\end{example}

\paragraph{A Procedure for Sampling Triangles.}

Each combinatorial type of tropical triangle is defined by a set of inequalities on its vertex set.  Thus, in order to sample triangles by type and assess their Alexandrov curvature, we need to specify the set of inequalities to define a sampling region, sample three random vertices from this region, draw the tropical line segments between these vertices, and compute the curvature of the resulting tropical triangle.  However, there are several observations to be mindful of when implementing such a procedure.

%\begin{example} \CA{FIX THIS EXAMPLE, GIVEN TRIANGLE IS FAT}
%\label{ex:T5_counterexample}
%Consider the vertex set
%$$
%a = (0,2), \qquad b = (2,0), \qquad c = (3,4).
%$$
%This is an isosceles T5 triangle with edge lengths $\dtr(a,b) = %\dtr(a,c) = 3$ and $\dtr(bc) = 4$.

%The tropical distance from the vertex $b$ to the $ac$ edge is
%$$
%\dtr((b_1, b_2),\, \gamma_{ac}(t)) = \begin{cases}
%3, & 0 \leq t \leq 1;\\
%2+t, & 1 \leq t \leq 2;\\
%4, & 2 \leq t \leq 3,
%\end{cases}
%$$
%while the Euclidean distance between the vertex $b'$ and the $a'c'$ edge on the comparison triangle is
%$$
%d_e(b',\, \gamma_{a'c'}(t)) = t^2 - \frac{2}{3}t + 9.
%$$

%\begin{figure}[h!]
%\centering
%\includegraphics[scale=0.3]{Figures/type5failz.png}
%\caption{Tropical distance function vs. Euclidean distance function %from Example~\ref{ex:T5_counterexample}}
%\label{fig:type5counterex}
%\end{figure}

%Comparing the tropical and Euclidean distances, at the endpoints where $t = 0$ and $t = a_1 - c_1$, we have equality between the tropical and Euclidean distances, but at the breakpoints over $\gamma_{ac}(t)$,
%$$
%3^2 < \frac{29}{3} \qquad \mbox{~and~} \qquad 4^2 > \frac{35}{3}.
%$$
%This change in sign is illustrated in Figure \ref{fig:type5counterex}.  Thus, Alexandrov curvature is not well-defined for this tropical triangle.
%\end{example}

%There are a couple of characteristics of these examples to note.  
Here and throughout the rest of the section we consider a tropical triangle in $\R^3/\R\one$ with vertices $a=(a_1,a_2), b= (b_1,b_2), c=(c_1,c_2)$ using the notational convention of Remark~\ref{rem:notation}.

The following set of inequalities on the coordinates of the vertices define a type T1 tropical triangle:
\begin{center}
\begin{tabular}{ccccc}
$a_1$ & $<$ &  $b_1$ & $<$ & $c_1$,\\
$a_2$ & $<$ & $b_2$ & $<$ & $c_2$,\\
$b_1 - b_2$ & $<$ & $a_1 - a_2$ & $<$ & $c_1 - c_2$.\\
\end{tabular}
\end{center}
Notice that Examples \ref{ex:eqskinny} and \ref{ex:T1_skinny} where Alexandrov curvature is well-defined study two different cases, defined by differing directions of inequality for the comparison of the differences $b_1 - a_1$ versus $c_2 - a_2$.  %Examples \ref{ex:T1_skinny} and \ref{ex:T1_counterexample} are both in the case when $b_1 - a_1 > c_2 - a_2$.
%, and \ref{ex:T5_counterexample} each

The following set of inequalities on the coordinates of the vertices define a type T5 tropical triangle:
\begin{center}
\begin{tabular}{ccccc}
$a_1$ & $<$ &  $b_1$ & $<$ & $c_1$,\\
$b_2$ & $<$ & $a_2$ & $<$ & $c_2$,\\
$a_1 - a_2$ & $<$ & $c_1 - c_2$ & $<$ & $b_1 - b_2$.\\
\end{tabular}
\end{center}
Similarly, the Examples \ref{ex:eqfat} and \ref{ex:T5_222} also study different cases defined by inequalities.  Contrary to triangle type T1 where there are only two cases, however, it turns out that there are six well-defined cases for triangle type T5.  These are characterized by the location of the vertex with respect to the bending point on its opposite edge, i.e., whether the vertices lie above or below, or before or after, the bending point on the opposite edge.  Thus, for each inequality set defining T1 triangles and T5 triangles, there are different cases to also encode in defining the sampling region for each triangle type.  The vertices then need to be permuted to take into account the issue of labeling when sampling the vertices randomly.

Furthermore, on each triangle type and case set, there are then planar isometries to take into account: rotations and reflections need to be considered.

Finally, note that the regions depicted by the triangle types, case sets, and planar isometries are unbounded.  To address this issue, we set a probability simplex constraint so that coefficients are nonnegative and sum to one.

These triangle vertex inequalities with the considerations described above describe a polyhedral region.  To sample from such a region, we implement the {\em hit and run algorithm} proposed by Boneh and Golan in 1979 and independently by Smith in 1980, which is designed to sample from bounded convex regions and for which there exists an \texttt{R} package \citep{boneh1979constraints,smith1980monte,TERVONEN2013552}.  The algorithm is a Markovian sampling technique that employs a random walk on the set defined by the input inequalities with a target distribution that is uniform over the set.  This particular choice of stationary distribution ensures that following the chain for a sufficiently large number of steps will result in the sampled points being approximately uniformly spread out over the set.

Table \ref{tab:har} shows the average occurrences of negatively and positively curved triangles and triangles with undefined curvature from sampled type T1 and T5 triangles for 10 runs each with 1000 randomly sampled triangles.

Our proposed procedure to sample tropical triangles by type is important and useful for further studies in random settings of so-called {\em transversal tropical triangles} by Ansola and de la Puente.  These correspond to Type 5 triangles and are of particular interest, because its vertices are tropically independent, and the coordinates of the vertices determine a tropically regular matrix and give triangulations of the tropical plane \citep{ansola2009note}.  Applying our proposed sampling procedure together with their results would therefore allow for random triangulations of the tropical plane.  %The classical Delaunay triangulation of a discrete set of points corresponds to the dual graph of the Voronoi diagram for the same set of points; tropical Voronoi diagrams have been recently studied by Criado and co-authors \cite{criado2021tropical}.  
Random triangulations of the tropical plane may increase the reach of tropical geometry in applications.  One particular application of interest corresponds to trees: classically, it is known that the Euclidean minimum spanning set tree of a set of points is a subset of the Delaunay triangulation of the same point and this relationship has been used to improve computational efficiency \citep{aurenhammer2013voronoi}.  Given the intimate connection between tropical geometry and trees, a tropical version of this relationship may be useful for similar computations of trees in the random setting, providing a bridge towards real data applications.

%The general expression for the tropical distance between the vertex $b$ and the $ac$ edge when $b_1 - a_1 < c_1 - b_1$ is
%\begin{equation}
%\label{eq:tr_dist_bac_T5}
%\dtr((b_1, b_2),\, \gamma_{ac}(t)) = \begin{cases}
%a_2 - b_2 + b_1 - a_1, & 0 \leq t \leq b_1 - a_1;\\
%a_2 + t - b_2, & b_1 - a_1 \leq t \leq c_2 - a_2;\\
%c_2 - b_2, & c_2 - a_2 \leq t \leq c_1 - a_1.
%\end{cases}
%\end{equation}

%Notice also that every triangle type except Type T2 has at least two cases, depending on the relative locations of the bending points of the tropical line segments connecting the three vertices of the triangle.  The cases give an additional inequality on each triangle type.  In this paper, we approach the study of tropical Alexandrov curvature based on the cases of each triangle type.

\begin{table}
\begin{center}
\caption{Sampled Type T1 and T5 Triangles and their Alexandrov Curvatures} \label{tab:har}      
\begin{tabular}{l|llll|l}
\hline\noalign{\smallskip}
Type & zero & positive & negative & undefined & total  \\
\noalign{\smallskip}\hline\noalign{\smallskip}
1 & 0\% & 0\% & 97\% & 3\% & 100\% \\
5 & 0\% & 14.3\% & 0\% & 85.7\% & 100\% \\
\noalign{\smallskip}\hline
\end{tabular}
\end{center}
\end{table}

%\CA{average of 10 runs each with 1000 random triangles}

\subsection{Fatness Everywhere: Triangle type T3}
\label{sec:flat_is_fat}

Up until now, we have established the existence of open regions for both negative and positive well-defined Alexandrov curvature.  We now give a definitive result for a setting where Alexandrov curvature is well-defined and positive (nonnegative) everywhere.

\begin{theorem}
\label{thm:flat_is_fat}
The tropical triangle of type T3 defined by the following set of inequalities always exhibits positive (nonnegative) Alexandrov curvature:
$$a_1 < b_1 < c_1, \qquad a_2 < b_2 < c_2, \qquad c_1 - c_2 < b_1 - b_2.$$
In other words, a tropical triangle of type T3 is always fat.
\end{theorem}

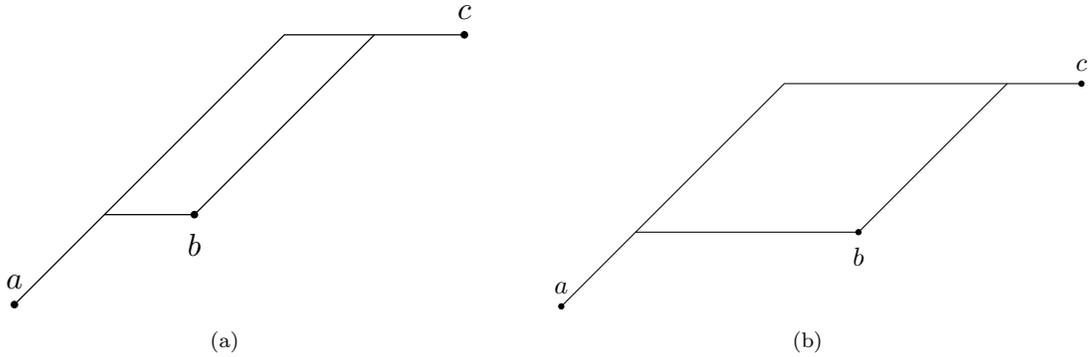
\begin{figure}[ht]
\centering
\subfigure[]{
\resizebox{0.4\textwidth}{!}{
\begin{tikzpicture}
\coordinate (a) at (0,0);
\coordinate (b) at (2,1);
\coordinate (c) at (5,3);
\coordinate (k1) at (3,3);
\coordinate (k2) at (1,1);
\coordinate (k3) at (4,3);
\foreach \coord in {a,b,c}
	\draw[fill=black] (\coord) circle[radius=1pt];
\foreach \coord in {a,c}
	\node[label={[label distance=-2pt]90:$\coord$}]at (\coord) {};
\foreach \coord in {b}
\node[label={[label distance=-20pt]:$\coord$}] at (\coord) {};
\draw (a)--(k1)--(c);
\draw (k2)--(b)--(k3);
\end{tikzpicture}}}
\hskip5mm
\subfigure[]{
\resizebox{0.45\textwidth}{!}{
\begin{tikzpicture}
\coordinate (a) at (0,0);
\coordinate (b) at (4,1);
\coordinate (c) at (7,3);
\coordinate (k1) at (3,3);
\coordinate (k2) at (1,1);
\coordinate (k3) at (6,3);
\foreach \coord in {a,b,c}
	\draw[fill=black] (\coord) circle[radius=1pt];
\foreach \coord in {a,c}
	\node[label={[label distance=-2pt]90:$\coord$}]at (\coord) {};
\foreach \coord in {b}
\node[label={[label distance=-20pt]:$\coord$}] at (\coord) {};
\draw (a)--(k1)--(c);
\draw (k2)--(b)--(k3);
\end{tikzpicture}}}
\caption{Two Cases for Triangle type T3: (a) $b_1 - a_1 < c_2 - a_2$; (b) $c_2 - a_2 < b_1 - a_1$.}
\label{fig:T3_cases} 
\end{figure}

\begin{proof}
We begin by considering the distance between the vertex $a$ and the $bc$ edge on the tropical triangle; note that the $bc$ edge is a tropical line segment of Type L1 given in Lemma \ref{lemma:tls_types}, so we compute
\begin{align*}
\dtr((a_1, a_2),\, \gamma_{bc}(t)) & = \begin{cases}
\dtr((c_1, c_2),\, (a_1 + t, a_2 + t)), & 0 \leq t \leq b_2 - a_2;\\
\dtr((c_1, c_2)),\, (a_1 + t, b_2), & b_2 - a_2 \leq t \leq b_1 - a_1.
\end{cases}\\
& = c_1 - a_1 - t,
\end{align*}
for all $0 \leq t \leq b_1 - a_1$.

Notice that for this type of tropical triangle, the Euclidean comparison triangle is degenerate: it is a single line where the vertices $a'$ and $c'$ are the endpoints, and the vertex $b'$ lies on the line between $a'$ and $c'$ at distance $b_1 - a_1$ from the vertex $a'$.

On the comparison triangle, the Euclidean distance between the vertex $a'$ and the $b'c'$ edge is $c_1 - a_1 - t$, which is exactly equal to the tropical distance between the vertex $a$ and the $bc$ edge and implies zero Alexandrov curvature.  Recall from Definition \ref{def:alexandrov} that the inequality (\ref{eq:alexandrovpos}) for fatness is weak, so here, it is also true that the tropical distance between the vertex $a$ and the $bc$ edge is greater than or equal to the Euclidean distance between the vertex $a'$ and the $b'c'$ edge.

A similar computation follows for the distance between the vertex $c$ and the $ab$ edge on the tropical triangle, $\dtr(c, \gamma_{ab}(t))$, by noting that the tropical line segment joining vertices $a$ and $b$ is again of Type L1 from Lemma \ref{lemma:tls_types}.  The tropical distance between the vertex $a$ and the $bc$ edge on the tropical triangle and the Euclidean distance between the vertex $a'$ and the $b'c'$ edge on the comparison triangle again coincide, implying zero Alexandrov curvature, or fatness of the tropical triangle when considering a weak inequality (\ref{eq:alexandrovpos}).

For the final vertex-to-edge comparison, $\dtr(b, \gamma_{ac}(t))$, notice that this triangle type has two cases, illustrated in Figure \ref{fig:T3_cases}.  We consider first consider the case illustrated in Figure \ref{fig:T3_cases}(a), defined by the inequality $b_1 - a_1 < c_2 - a_2$.  The tropical distance between the vertex $b$ and any point on the $ac$ edge is quite a bit more complicated; it is given by
\begin{equation*}
\dtr((b_1, b_2),\, \gamma_{ac}(t)) = \begin{cases}
b_1 - a_1 - t, & 0 \leq t \leq b_2 - a_2;\\
b_1 - b_2 + a_2 - a_1, & b_2 - a_2 \leq t \leq b_1 - a_1;\\
a_2 - b_2 + t, & b_1 - a_1, \leq t \leq c_2 - a_2;\\
c_2 - b_2, & c_2 - a_2 \leq t \leq b_1 + c_2 - b_2;\\
a_1 - b_1 + t, & b_1 + c_2 - b_2 \leq t \leq c_1 - a_1.
\end{cases}
\end{equation*}
The tropical to Euclidean distance comparison is then $\dtr((b_1, b_2),\, \gamma_{ac}(t))$ to
\begin{equation}
\label{eq:T3_de}
\begin{aligned}
d_e((b_1', b_2'),\, \gamma_{a'c'}(t)) & = \begin{cases} 
b_1 - a_1 - t, & 0 \leq t \leq b_1 - a_1;\\
a_1 - b_1 - t, & b_1 - a_1 \leq t \leq c_1 - a_1;
\end{cases}\\
& = |b_1 - a_1 - t|.
\end{aligned}
\end{equation}
We have equality of the tropical and Euclidean distances at the endpoints where $t = 0$ and $t = c_1 - a_1$, and
\begin{align*}
b_1 - b_2 + a_2 - a_1 \geq b_1 - a_1 - t,\\ 
a_2 - b_2 + t \geq a_1 - b_1 + t,\\
c_2 - b_2 > a_1 - b_1 + t,
\end{align*}
using the fact that we are in the case when $b_1 - a_1 < c_2 - a_2$.  Thus, we conclude that the tropical Alexandrov curvature is positive, and we have fatness (in comparison to the Euclidean case) when considering the distance between the vertex $b$ and the $ac$ edge of the tropical triangle.

We now turn to the case illustrated in Figure \ref{fig:T3_cases}(b), defined by the inequality $c_2 - a_2 < b_1 - a_1$, and compute the tropical distance between the vertex $b$ and any point on the $ac$ edge:
\begin{equation*}
\dtr((b_1, b_2),\, \gamma_{ac}(t)) = \begin{cases}
b_1 - a_1 - t, & 0 \leq t \leq b_2 - a_2;\\
b_1 - b_2 + a_2 - a_1, & b_2 - a_2 \leq t \leq c_2 - a_2;\\
c_2 - b_2 + b_1 - a_1 - t, & c_2 - a_2 \leq t \leq b_1 - a_1;\\
c_2 - b_2, & b_1 - a_1 \leq t \leq b_1 + c_2 - b_2 - a_1;\\
a_1 - b_1 + t, & b_1 + c_2 - b_2 - a_1 \leq t \leq c_1 - a_1.
\end{cases}
\end{equation*}
As above, we compare this to the comparison triangle where we take the Euclidean distance between the vertex $b'$ and the $a'c'$ edge given by (\ref{eq:T3_de}).  Again, we have equality at the endpoints, and
\begin{align*}
b_1 - b_2 + a_2 - a_1 \geq b_1 - a_1 - t,\\
c_2 - b_2 + b_1 - a_1 - t \geq b_1 - a_1 - t,\\
c_2 - b_2 > a_1 - b_1 + t,
\end{align*}
using the fact that we are in the case when $c_2 - a_2 < b_1 - a_1$.  Thus, we again have fatness when considering the distance between the vertex $b$ and the $ab$ edge on the tropical triangle in this second case, which completes the proof that the Alexandrov curvature of tropical triangles of type T3 is always positive (nonnegative) and so these triangles are always fat.
\end{proof}

%\CA{Graph corresponds to triangle with vertices $(0,0),(3,3),(8,3)$}
Visually, fatness corresponds to the curves of the tropical distances of the tropical triangles always lying above, or coinciding with, the curves of the Euclidean distances of the comparison triangles.  Here, on two of the three vertex to edge comparison, the two curves coincide, while the final $b$ to $ac$ edge coincides with the Euclidean distances at the beginning and end segments of the range of $t$, and in the middle, the curves of the tropical distance lie above the Euclidean distances.  This is illustrated in Figure \ref{fig:type3}, with an example triangle given by vertices $a = (0,0),\, b = (3,3),\, c = (8,3)$.

\begin{figure}[h!]
\centering
\includegraphics[scale=0.3]{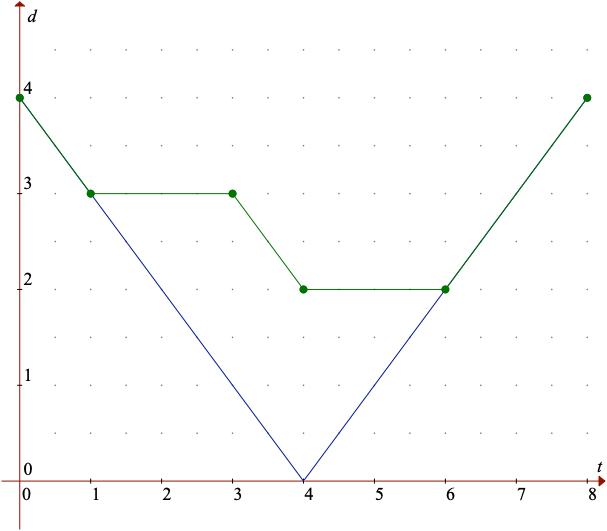}
\caption{Typical tropical distance function vs. Euclidean distance function for Type T3.}
\label{fig:type3}
\end{figure}

\begin{remark}\label{rmk:flat}
Flat tropical triangles do exist in $\R^n/\R\one$, albeit in degenerate form when the three vertices lie on the same tropical line. It is worth mentioning that, while sampling uniformly from the simplex as in the last section would give probability zero to the occurrence of three sampled vertices being tropically collinear, this can quickly change if the distribution is different. Say that for an application in question, we only consider vertices with integer coordinates and $a,b,c \in \R^2/\R\one$ are drawn uniformly from the integers $\{0,\dots,10\}$. In one experiment, we drew 100 triangles in this way and computed their curvature, obtaining the results of Table~\ref{tab:intcoord}. We note that now 20\% of the triangles were flat. One such random triangle with $a=(5,1)$, $b=(7,3)$ and $c=(10,4)$ is shown in Figure~\ref{fig:flat}.
\begin{table}
\begin{center}
\caption{Proportions of sampled triangles with integer coordinates according to their curvature} \label{tab:intcoord}      
\begin{tabular}{llll|l}
\hline\noalign{\smallskip}
zero & positive & negative & undefined & total  \\
\noalign{\smallskip}\hline\noalign{\smallskip}
20\% & 22\% & 7\% & 51\% & 100\% \\
\noalign{\smallskip}\hline
\end{tabular}
\end{center}
\end{table}
\end{remark}

\begin{figure}[h!]
\centering
\includegraphics[scale=0.3]{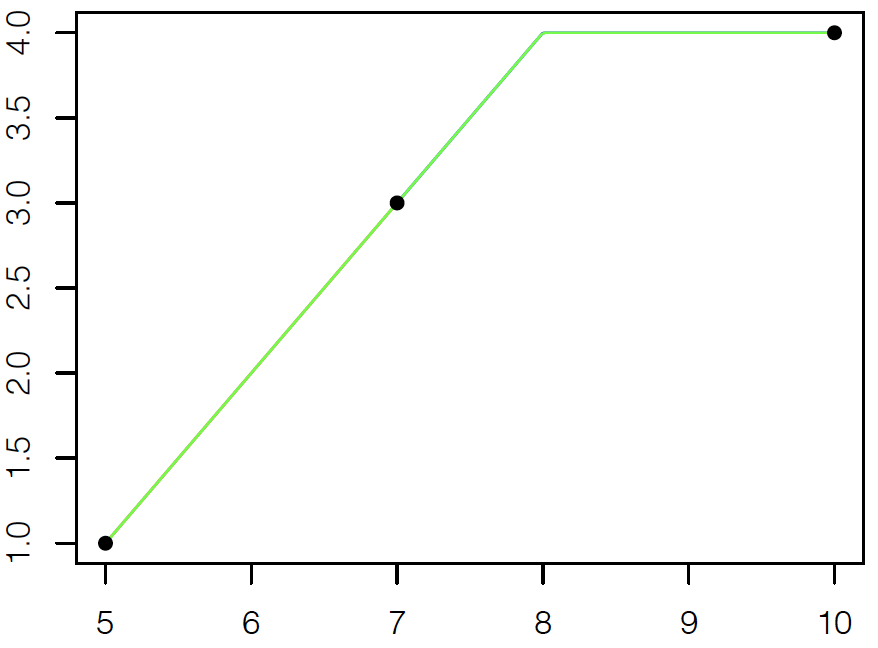}
\caption{Tropical triangle obtained by random integer sampling with zero curvature (flat).}
\label{fig:flat}
\end{figure}

%\begin{remark}
%\label{remark:flat_is_flat}
%These tropical triangles fail to be flat because the tropical line segment between $a$ and $c$ is different than the concatenation of the two tropical line segments with endpoints as vertices $a$ and $c$ and the join as vertex $b$. However, this concatenation is also a valid geodesic. If we were to consider a triangle with this modified side, then we would actually obtain the tropical equivalent of the degenerate Euclidean comparison triangle. In this case, we have zero Alexandrov curvature and the tropical triangle would be flat. 
%\end{remark}

%%%%%%%%%%%%%%%%%%%%%%%%%%%%%%%%%%%%%%%%%%%%%%%%%%%

\section{Undefined Alexandrov Curvature in $\R^3/\R\one$}
\label{sec:undefined}

While in the previous Section \ref{sec:flat_is_fat} we establish that Alexandrov curvature can be well-defined and positive (nonnegative) everywhere on the same combinatorial type, we now show that Alexandrov curvature can also be {\em never} well-defined on a combinatorial type. This happens if we use triangle types T2 and T4 to study Alexandrov curvature, illustrated in Figure \ref{fig:5types}(b) and (c), respectively.

As we will see in this section and later on, undefined Alexandrov curvature is a common occurrence in the tropical projective torus.  In particular, it also occurs in the space of phylogenetic trees.  This is an important fact to take into account when developing statistical methodology that depends on geometric features of the underlying space.

\subsection{Undefined Curvature Everywhere: Triangle Types T2 and T4}

We now prove that the combinatorial triangle types T2 and T4 always have undefined Alexandrov curvature. We consider the vertices to be $a=(a_1,a_2)$, $b=(b_1,b_2)$ and $c=(c_1,c_2)$  in $\R^3/\R\one$.

\begin{theorem}\label{thm:T2nocurvature}
Consider a general triangle of type T2 given by the inequalities
$$a_1 < b_1 < c_1, \qquad a_2 < c_2 < b_2, \qquad a_1-a_2 < b_1 - b_2.$$
Then any such triangle has no well-defined Alexandrov curvature.
\end{theorem}
\begin{proof}
First we see that the lengths of the sides of such a triangle are given by Lemma \ref{lemma:tls_types}
\begin{align*}
    \dtr((a_1,a_2),\, (b_1, b_2)) & = b_1 - a_1 \\
    \dtr((a_1,a_2),\, (c_1, c_2)) & = c_1 - a_1 \\
    \dtr((b_1,b_2),\, (c_1, c_2)) & = c_1 - b_1 + b_2 - c_2
\end{align*}
Now we compute the distance of vertex $c$ to the opposite side $ab$ to obtain
\begin{equation*}
\dtr((c_1,c_2),\, \gamma_{ab}(t)) = \begin{cases}
c_1-a_1-t, & 0 \leq t \leq c_2-a_2,\\
c_1-a_1+a_2-c_2, & c_2-a_2 \leq t \leq b_2-a_2,\\
c_1-a_1-t+b_2-c_2, & b_2-a_2 \leq t \leq b_1-a_1.
\end{cases}
\end{equation*}
On the other hand, by Lemma \ref{lem:euclidist}, the Euclidean distance is given by
$$ d_e^2(c',\gamma_{ab}(t)) = t^2 + \frac{(c_1-b_1+b_2-c_2)^2-(c_1-a_1)^2-(b_1-a_1)^2}{b_1-a_1}t + (c_1-a_1)^2$$
Consider the beginning of the segment, when $0 \leq t \leq c_2-a_2$. Then
\begin{align*}
    d_e^2(c',\gamma_{ab}(t)) -& \dtr^2(c,\, \gamma_{ab}(t))  \, = \, d_e^2(c',\gamma_{ab}(t)) - ((c_1-a_1)-t)^2 \\
    &= \frac{(c_1-b_1+b_2-c_2)^2-(c_1-a_1)^2-(b_1-a_1)^2}{b_1-a_1}t +2(c_1-a_1)t \\
    &= \frac{(c_1-b_1+b_2-c_2)^2-((c_1-a_1) - ( b_1-a_1))^2}{b_1-a_1}t \\
    &= \frac{(b_2-c_2)(2(c_1-b_1) + b_2 - c_2)}{b_1-a_1}t \, \, > \, 0
\end{align*}
where we have used that $b_2-c_2>0$ , \, $c_1-b_1>0$, \,  $b_1-a_1>0$ and $t>0$.
%$$(c_1-b_1) + (c_1-c_2-b_1+b_2) > c_1-c_2-b_1+b_2 > b_1 $$
If we now consider the end of the segment, when $b_2 - a_2 \leq t \leq b_1-a_1$, then (using part of the above calculation)
\begin{align*}
   d_e^2(c'&,\gamma_{ab}(t)) - \dtr^2(c,\, \gamma_{ab}(t))  \, = \, d_e^2(c',\gamma_{ab}(t)) - ((c_1-a_1-t)+b_2-c_2)^2 \\ 
   & = \frac{(b_2-c_2)(2(c_1-b_1) + b_2 - c_2)}{b_1-a_1}t - 2(c_1-a_1-t)(b_2-c_2)-(b_2-c_2)^2 \\
   & = (b_2-c_2) \frac{(2c_1-2a_1 + b_2 - c_2)t - (2(c_1-a_1)+(b_2-c_2))(b_1-a_1)}{b_1-a_1} \\
   & = - \frac{(b_2-c_2)(2(c_1-a_1) + b_2 - c_2)(b_1-a_1-t)}{b_1-a_1} \, < \,0
\end{align*}
where we have used that $b_2-c_2>0$, \, $c_1-a_1>0$, \, $b_1-a_1>0$ and $b_1-a_1-t > 0$.
Therefore, the comparison between the lengths of the Euclidean and tropical segments is not consistent throughout the whole side, and hence Alexandrov curvature for these triangles is undefined.
\end{proof}

Visually, the undefined Alexandrov curvature corresponds to a change in the relative positions of the curves of the tropical and Euclidean distances; an example with vertex set $a = (0,0),\, b = (3,2),\, c = (4,1)$ and the distances for the $b$ $(b')$ to $ac$ $(a'c')$ edge is shown in Figure \ref{fig:type2}.

%\CA{Example with vertices $(0,0),(3,2),(4,1)$}
\begin{figure}[h!]
\subfigure[]{
\centering
\includegraphics[scale=0.28]{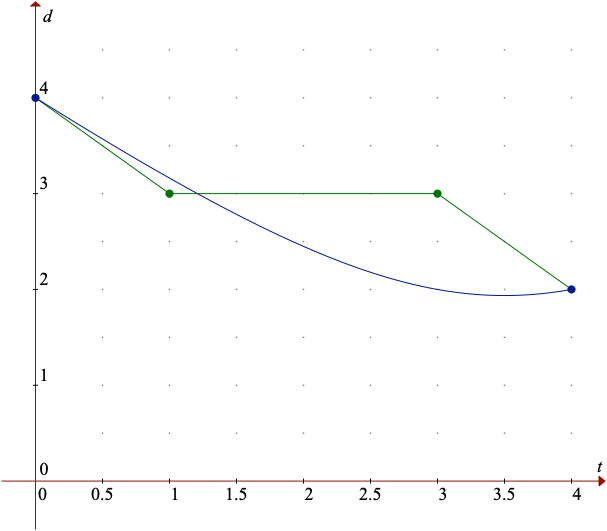}}
\hskip35mm
\subfigure[]{
\begin{tikzpicture}
\coordinate (o) at (0,0);
\coordinate (a) at (0,0);
\coordinate (b) at (4,2.6);
\coordinate (c) at (5.3,1.3);
\coordinate (k1) at (2.6,2.6);
\coordinate (k2) at (4,1.3);
\coordinate (k3) at (1.3,1.3);
\foreach \coord in {a,b,c}
	\draw[fill=black] (\coord) circle[radius=1pt];
\foreach \coord in {a,b,c}
	\node[label={[label distance=-2pt]90:$\coord$}]at (\coord) {};
\draw (a)--(k1);
\draw (b)--(k2);
\draw (c)--(k3);
\draw (k1)--(b);
\end{tikzpicture}}
\caption{(a) Typical tropical distance function vs. Euclidean distance function for type T2; (b) Tropical triangle T2.}
\label{fig:type2}
\end{figure}
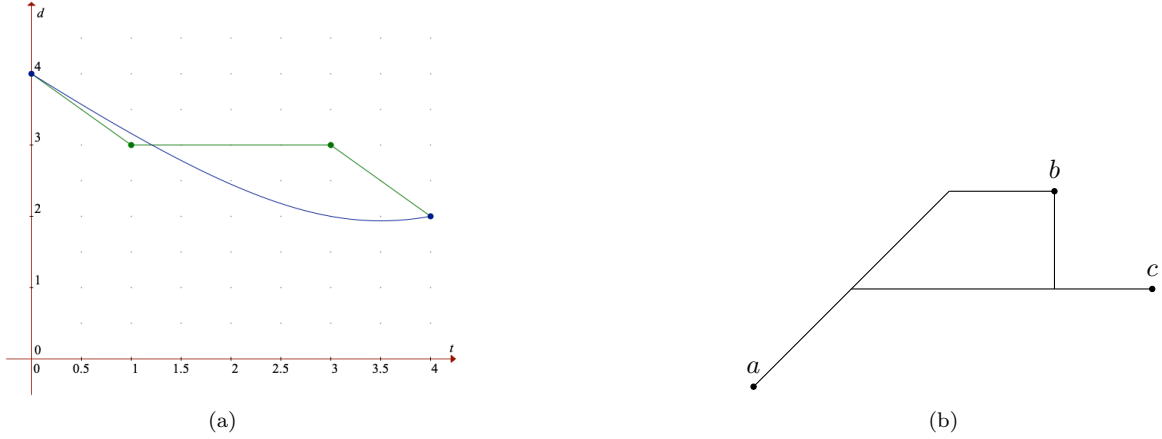

We see a similar behavior for T4 triangles, defined by the following set of inequalities:
$$
a_1 < b_1 < c_1, \qquad b_2 < a_2 < c_2, \qquad a_1 - a_2 < c_1 - c_2, \qquad b_1 - b_2 < c_1 - c_2.
$$
%Notice that this triangle can never be equilateral, since the edge lengths given by Lemma \ref{lemma:tls_types} are $c_1 - a_1$ for the $ac$ edge and $c_1 - b_1$ for the $bc$ edge, but $a_1 > b_1$.

To compute the tropical distance function between vertex $b$ and the $ac$ edge we distinguish between two cases,
$$
c_2 - a_2 \leq b_1 - a_1 \qquad \mbox{~and~} \qquad c_2 - a_2 \geq b_1 - a_1,
$$
illustrated in Figure \ref{fig:T4_cases}.

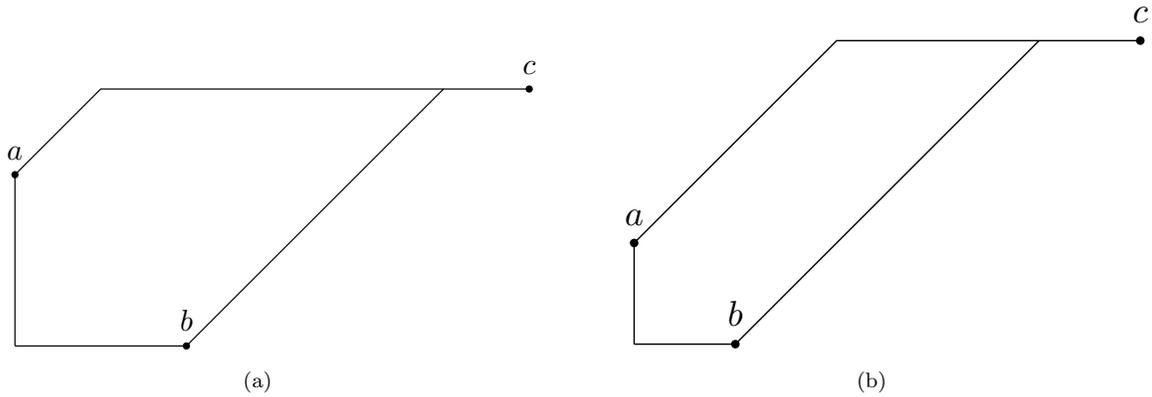
\begin{figure}[ht]
\centering
\subfigure[]{
\resizebox{0.45\textwidth}{!}{
\begin{tikzpicture}
\coordinate (a) at (0,2);
\coordinate (b) at (2,0);
\coordinate (c) at (6,3);
\coordinate (k1) at (0,0);
\coordinate (k2) at (1,3);
\coordinate (k3) at (5,3);
\foreach \coord in {a,b,c}
	\draw[fill=black] (\coord) circle[radius=1pt];
\foreach \coord in {a,b,c}
	\node[label={[label distance=-2pt]90:$\coord$}]at (\coord) {};
\draw (k1)--(a)--(k2)--(c);
\draw (k1)--(b)--(k3);
\end{tikzpicture}}}
\hskip5mm
\subfigure[]{
\resizebox{0.45\textwidth}{!}{
\begin{tikzpicture}
\coordinate (a) at (0,1);
\coordinate (b) at (1,0);
\coordinate (c) at (5,3);
\coordinate (k1) at (0,0);
\coordinate (k2) at (2,3);
\coordinate (k3) at (4,3);
\foreach \coord in {a,b,c}
	\draw[fill=black] (\coord) circle[radius=1pt];
\foreach \coord in {a,b,c}
	\node[label={[label distance=-2pt]90:$\coord$}]at (\coord) {};
\draw (k1)--(a)--(k2)--(c);
\draw (k1)--(b)--(k3);
\end{tikzpicture}}}
\caption{Two cases for triangle type T4: (a) $c_2 - a_2 \leq b_1 - a_1$; (b) $c_2 - a_2 \geq b_1 - a_1$.}
\label{fig:T4_cases} 
\end{figure}

For example, the general expression for the tropical distance between the vertex $b$ and the $ac$ edge when $c_2 - a_2 \leq b_1 - a_1$ is
\begin{equation*}
\label{eq:tr_dist_bac_T4}
\dtr((b_1, b_2),\, \gamma_{ac}(t)) = \begin{cases}
a_2 - b_2 + b_1 - a_1, & 0 \leq t \leq c_2 - a_2;\\
c_2 - b_2 + b_1 - a_1 - t, & c_2 - a_2 \leq t \leq b_1 - a_1;\\
c_2 - b_2, & b_1 - a_1 \leq t \leq c_2 - b_2 + b_1 - a_1;\\
a_1 + t - b_1, & c_2 - b_2 + b_1 - a_1 \leq t \leq c_1 - a_1.
\end{cases}
\end{equation*}
The corresponding expression for the case $c_2 - a_2 \geq b_1 - a_1$ is the same at the initial and final parts of the segment, and just varies slightly in the middle.

\begin{theorem}\label{thm:T4nocurvature}
Consider a general triangle of type T4 given by the inequalities
$$
a_1 < b_1 < c_1, \qquad b_2 < a_2 < c_2, \qquad a_1 - a_2 < c_1 - c_2, \qquad b_1 - b_2 < c_1 - c_2.
$$
Then any such triangle has no well-defined Alexandrov curvature.
\end{theorem}
\begin{proof}
First we use Lemma \ref{lemma:tls_types} to compute the lengths of the sides of such a triangle. We have that $ab$ is of type L3, while $ac$ and $bc$ are of type L1, so
\begin{align*}
    \dtr((a_1,a_2),\, (b_1, b_2)) & = a_2 - b_2 + b_1 - a_1 \\
    \dtr((a_1,a_2),\, (c_1, c_2)) & = c_1 - a_1 \\
    \dtr((b_1,b_2),\, (c_1, c_2)) & = c_1 - b_1.
\end{align*}
We consider the distance of vertex $c$ to the opposite side $ab$. As mentioned in the paragraph preceding the theorem, for small enough $t>0$ this is given by
$\dtr(b,\, \gamma_{ac}(t)) = a_2 - b_2 + b_1 - a_1$, while for large enough $t<c_1-a_1$ it is given by $a_1-b_1+t$.

On the other hand, by Lemma \ref{lem:euclidist}, the square of the Euclidean distance $d_e^2(b',\gamma'_{ab}(t))$ is given by
$$t^2 + \frac{(c_1-b_1)^2-(a_2-b_2+b_1-a_1)^2-(c_1-a_1)^2}{c_1-a_1}t + (a_2-b_2+b_1-a_1)^2.$$

At the beginning of the segment, we have
\begin{align*}
    \dtr^2(c',\gamma_{ab}(t)) - d_e^2(c&,\,  \gamma_{ab}(t))   \\
    &= t \left( -\frac{(c_1-b_1)^2-(a_2-b_2+b_1-a_1)^2-(c_1-a_1)^2}{c_1-a_1}-t \right) 
\end{align*}
Since $t>0$ is small enough, we can use $t\leq b_1-a_1$ to bound the second factor from below by
\begin{align*}
  & -\frac{(c_1-b_1)^2-(a_2-b_2+b_1-a_1)^2-(c_1-a_1)^2}{c_1-a_1}-(b_1-a_1)  \\
  &= \frac{(a_2-b_2+b_1-a_1)^2-(c_1-b_1)^2}{c_1-a_1}+(c_1-a_1)-(b_1-a_1) \\
   &= \frac{(a_2-b_2+b_1-a_1)^2-(c_1-b_1)^2+(c_1-b_1)(c_1-a_1)}{c_1-a_1} \\
     &= \frac{(a_2-b_2+b_1-a_1)^2+(c_1-b_1)(b_1-a_1)}{c_1-a_1} \, > \, 0
\end{align*}
where we have used that $c_1-a_1>0$ , \, $c_1-b_1>0$ \, and \, $b_1-a_1>0$. Since $t>0$ as well, we conclude that 
 $$\dtr^2(c',\gamma_{ab}(t)) - d_e^2(c,\,  \gamma_{ab}(t)) > 0 \qquad \text{ for } t\leq \min \{ c_2-a_2, b_1-a_1 \}.$$
If we now consider the end of the segment, by a similar calculation as in the last part of the proof of Theorem~\ref{thm:T2nocurvature},
\begin{align*}
   d_e^2(c',\gamma_{ab}(t)) - \dtr^2(c,\, \gamma_{ab}&(t))  \, = \, d_e^2(c',\gamma_{ab}(t)) - (t-(b_1-a_1))^2 \\ 
   & = - \frac{(a_2-b_2)(2(b_1-a_1) + a_2 - b_2)(c_1-a_1-t)}{c_1-a_1} \, < 0.
\end{align*}
where we have used that $a_2-b_2>0$, \, $b_1-a_1>0$, \, $c_1-a_1>0$ and $c_1-a_1-t > 0$.
This sign change implies that Alexandrov curvature for these triangles is undefined.
\end{proof}

For example, consider the vertices $a = (3,4),\, b = (6,3),\, c = (9,5)$ in $\R^3/\R\one$ with distance curves for the $b$ $(b')$ to $ac$  $(a'c')$ edge illustrated in Figure \ref{fig:type4}.

%\CA{Taken from the commented example!}
\begin{figure}[h!]
\centering
\includegraphics[scale=0.3]{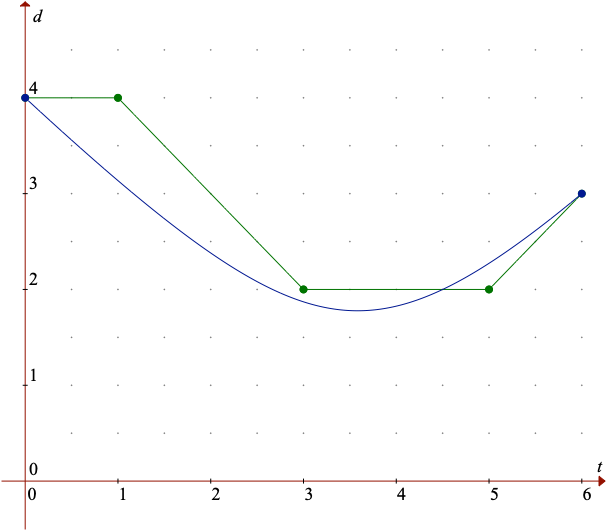}
\caption{Typical tropical distance function vs. Euclidean distance function for type T4}
\label{fig:type4}
\end{figure}

Undefined Alexandrov curvature is an unexpected and surprising phenomenon, especially in the ``nice'' setting of the tropical projective torus, which is a metric measure space that satisfies all the properties necessary for probabilistic questions and statistical methodology to be well-defined.  In general, undefined Alexandrov curvature is a peculiar phenomenon: it is known to occur in branching and bifurcating spaces, and particularly when geodesics bifurcate \citep{plaut2001metric}.  For this reason, Alexandrov curvature is not an appropriate measure of curvature to use on skeletal spaces such as graphs or trees.  This, however, is not the setting of $(\R^n/\R\one, \dtr)$ and, to the best of our knowledge, is the first occurrence of undefined Alexandrov curvature in a well-defined, continuous, metric measure space that is moreover a vector space (see Remark \ref{rem:vectorspace}).  This invites a further in-depth exploration of this phenomenon; potentially, the underlying relation between the tropical projective torus and phylogenetic tree space, as well as the polyhedral nature of tropical geometry, may be contributing to the undefinedness.  It also raises the question that perhaps undefined Alexandrov curvature only occurs in the subset of the tropical projective torus that corresponds to tree space.  As we will see in the next section, this is not the case, so the answer is not so easy.

\subsection{How often does each triangle type occur?}

We have so far established by proof that in the plane, two out of the five combinatorial types of tropical triangle always have undefined curvature and one always has positive curvature.  We now show by numerical experimentation how often Alexandrov curvature is well-defined and give the breakdown by triangle type.

\begin{table}
\begin{center}
\caption{Average proportions of sampled triangles types} \label{tab:types}      
\begin{tabular}{lllll|l}
\hline\noalign{\smallskip}
Type 1 & Type 2 & Type 3 & Type 4 & Type 5 & Total  \\
\noalign{\smallskip}\hline\noalign{\smallskip}
11.1\% & 35.9\% & 21.2\% & 25.8\% & 6\% & 100\% \\
\noalign{\smallskip}\hline
\end{tabular}
\end{center}
\end{table}

Based on these numerical results, an interesting question is what is the ``size'' of the set of each triangle type and each curvature type? This question is of independent interest apart from the Alexandrov curvature implication. We coded a function that takes an arbitrary triangle in the plane and returns its type (of which there could be multiple, if the triangle is not generic). As with our other experiments, we sampled uniformly from the probability simplex. Table~\ref{tab:types} shows the average proportions over 10 runs for each of the five types with a sample of 1000 triangles each time. 
Notice that, indeed, these results are consistent with Tables~\ref{tab:har} and the first column of \ref{tab:experiments_n}.

%It was gratifying to observe that the results are consistent with Tables~\ref{tab:har} and the first column of \ref{tab:experiments_n}.

Another way that this problem may be phrased is in terms of tropical volume: using existing results by Loho \& Schymura, the volume of these sets could possibly be estimated \citep{Loho2020}.

%%%%%%%%%%%%%%%%%%%%%%%%%%%%%%%%%%%%%%%%%%%%%%%%%%%

\section{Alexandrov Curvature in Higher Dimensions}
\label{sec:higherdim}

While we have focused on the tropical plane, the question of Alexandrov curvature also makes sense in higher dimensions for a general $n$, i.e., in $\R^n/\R\one$. 

We prove a generalization of Theorem~\ref{thm:flat_is_fat} for a special type of triangles in $n$ dimensions. We follow the notational convention from Remark~\ref{rem:notation}.

\begin{theorem}\label{thm:flat_is_fat_n}
Any $n$-dimensional tropical triangle given by $a,b,c \in \R^{n+1}/\R\one$ such that
\begin{equation*}
    a_1 < b_1 < c_1 \,, \, a_2 <b_2 <c_2 \,, \,\dots,\, a_{n-1} < b_{n-1} < c_{n-1} \,,\, a_n > b_n > c_n \label{eq:super3_1}
\end{equation*}
and 
\begin{align}
    b_1 - a_1 > b_2 -a_2 > &\dots > b_{n-1}-a_{n-1} \nonumber\\
    c_1 - b_1 > c_2 -b_2 > &\dots > c_{n-1}-b_{n-1} \label{eq:super3_2}
\end{align}
is positively curved (fat).
\end{theorem}

\begin{proof}
We first note that adding the inequalities in \eqref{eq:super3_2} we obtain
\begin{equation*}\label{eq:super3_3}
    c_1 - a_1 > c_2 -a_2 > \dots > c_{n-1}-c_{n-1}
\end{equation*}
so that we have
\begin{align*}
    \dtr(a,\, b) & = b_1 - a_1 + a_n - b_n, \\
    \dtr(a,\, c) & = c_1 - b_1 + b_n - c_n, \\
    \dtr(b\, c) & = c_1 - a_1 + a_n - c_n.
\end{align*}
Since $\dtr(a,\, b) + \dtr(a,\, c) = \dtr(b\, c)$, the comparison Euclidean triangle consists of three collinear points $a', b', c'$. 

Let us consider the tropical segment between $b$ and $c$, which is now formed by $n$ line segments. This can be parametrized by
\begin{small}
$$\gamma_{bc}(t) = 
\begin{cases}
(b_1,\dots,b_{n-1},b_n-t) & 0 \leq t \leq b_n-c_n;\\
(b_1+c_n-b_n+t,\dots,b_{n-1}+c_n-b_n+t,c_n)  & b_n-c_n \leq t \leq b_n-c_n+c_{n-1}-b_{n-1};\\
\vdots & \vdots \\
(b_1+c_n-b_n+t,c_2,\dots,c_n) , & b_n-c_n+c_2-b_2 \leq t \leq b_n-c_n+c_1-b_1
\end{cases}
$$
\end{small}
In spite of its complexity, it can be verified that the tropical distance from the vertex $a$ to $\gamma_{bc}(t)$ is consistently given by
$$\dtr(a,\, \gamma_{bc}(t)) = b_1-a_1 + a_n - b_n + t =  d_e(a'\,, \gamma'{bc}(t)) $$
for $0\leq t \leq b_n-c_n+c_1-b_1$. Analogously, $\dtr(c,\, \gamma_{ab}(t)) = d_e(c'\,, \gamma'{ab}(t))$ so it remains to compare the distance from $b$ to $\gamma_{ac}(t)$. The latter has the same parametrization as above just replacing the role of $b$ by $a$. With this we have
\begin{small}
$$\dtr(b,\, \gamma_{ac}(t)) = 
\begin{cases}
b_1-a_1-(b_n-a_n+t) & 0 \leq t \leq a_n-b_n;\\
b_1-a_1 & a_n -b_n \leq t \leq \min\{a_n-c_n, b_1-a_1+a_n-b_n\};\\
\vdots & \vdots \\
c_2-b_2+b_n-c_n &  a_n-c_n+c_2-a_2 \leq t  \\ 
a_1-a_n-b_1+b_n+t & c_2-b_2+b_1-a_1+a_n-c_n \leq t \leq a_n-c_n+c_1-a_1 \\
\end{cases}
$$
\end{small}
where the exact expressions in the middle part depend on whether additional inequalities hold (see the proof of Theorem~\ref{thm:flat_is_fat}). However, we do not need to know the middle behavior to conclude that one distance function is always greater than the other. Indeed, note that at the beginning of the segment and at the end of the segment they coincide, and then the tropical distance function becomes constant in the adjacent segments, while the euclidean distance function $d_e(b',\gamma'_{ac}(t)) = |b_1-a_1+a_n-b_n-t|$ keeps decreasing linearly towards 0. Since the tropical distance can also move at most at a linear rate, this means that inside the segment it remains strictly larger than the euclidean one. Hence,
$$\dtr(b,\, \gamma_{ac}(t)) > d_e(c'\,, \gamma'{ab}(t)) $$ holds for all $0 \leq t \leq a_n-c_n+c_1-a_1$
and we can conclude positive curvature.
\end{proof}

While it was possible to generalize the triangle type T3 to higher dimensions as done in Theorem \ref{thm:flat_is_fat_n}, in general it is quite difficult to study triangle types in higher dimensions.  This is because the number of combinatorial types of triangles quickly explodes even for relatively low dimensions: in four dimensions, the number of combinatorial types of triangle grows to 35; in five dimensions, it is 530; and in six dimensions, the number explodes to 13,621; in higher dimensions, the number of combinatorial types is unknown \citep{maclagan2015introduction}.  In particular, a general formula for the number of combinatorial types of tropical triangles for an arbitrary $n$ is currently unknown.

We were, however, able to study the curvature behavior by numerical experimentation in higher dimensions: in particular, we found higher dimensional examples exhibiting all the possible curvature behaviors. To understand this better, we implemented Algorithm 1 of \cite{yoshida2022tree} to compute tropical line segments in $n$ dimensions. This allows us to extend our curvature computations to any triangle in $a,b,c \in \R^n/\R\one$. For several values of $n$, we sampled 1000 random triangles uniformly from the $n-1$ dimensional simplex, and computed the proportion of flat, positive, negative and undefined curvature. We repeated this experiment ten times for each dimension. The results are presented in Table \ref{tab:experiments_n}.

\begin{table}
\begin{center}
\caption{Percentage of sampled triangles in $\R^n/\R\one$ according to their Alexandrov curvature}
\label{tab:experiments_n}       
\begin{tabular}{c|cccccc}
\hline\noalign{\smallskip}
$n$ & 3 & 4 & 5 & 6 & 7 & 8 \\
\noalign{\smallskip}\hline\noalign{\smallskip}
$0$ & 0 & 0 & 0 & 0 & 0 & 0 \\
$+$ & 26.3 & 10.9 & 6.1 & 3.5& 2.1& 1.2\\
$-$ & 10.8 & 6.4 & 5.4 & 4.7 & 3.9 & 3.4\\
$\nexists$ & 62.9 & 82.7 & 88.5 & 91.8 & 94 & 95.4 \\
\noalign{\smallskip}\hline
\end{tabular}
\end{center}
\end{table}

We observe that the proportions of positively and negatively curved triangles decrease with $n$, while the proportion of triangles with undefined curvature steadily increases to conform the vast majority of the space. On the other hand, no flat triangles were sampled, consistent with their expected probability zero.

\subsection{Curvature in Ultrametric Tree Space}

We now present an example of Alexandrov curvature computed on sets of phylogenetic trees, which was the original motivation of our study (see Section \ref{sec:motivation}).  Assessing the curvature behavior on the space of phylogenetic trees will determine whether Sturm's algorithm to compute Fr\'{e}chet means will be applicable or not.

Briefly, a tree is an acyclic connected graph $T = (V,E)$ defined by a set of vertices $V$ and a set of edges $E$.  An $N$-tree is a tree with $N$ labeled terminal nodes called leaves.  A metric $N$-tree is a tree with zero or positive lengths on all its edges; metric $N$-trees are also called {\em phylogenetic trees}.  {\em Ultrametric trees} are special cases of phylogenetic trees where the distance from the unique root to a leaf is the same among all leaves; ultrametric trees are also known as equidistant trees.  The set of all ultrametric trees is a proper subset of the space of phylogenetic trees and therefore also contained within the tropical projective torus; e.g., \cite{monod2018tropical,lin2022tropical}.  An ultrametric tree is represented as a vector in $\R^{\binom{N}{2}}$, where each entry in the vector is the total distance between leaf $i$ and leaf $j$, given by summing all the branch lengths along the unique path between the two leaves.

We simulated sets of ultrametric trees with branch lengths only slightly perturbed.  Our reasoning to do so relates back to our original motivation of doing statistics, as discussed above in Section \ref{sec:motivation}: ultimately, we would like to do statistics with sets of phylogenetic trees and in particular, we would like to compute and study Fr\'{e}chet means of sets of phylogenetic trees.  An important reason for studying Fr\'{e}chet means, aside from it being a fundamental statistic, is that in the current literature, means of sets of phylogenetic trees are known to behave poorly.  Specifically, means of sets of phylogenetic trees computed using the current standard are known to be {\em sticky}, meaning that different sets of different input trees results in no change in the mean \citep{10.1214/12-AAP899}.  This property is prohibitive for parametric statistical inference: it means that classically-derived probabilistic central limit theorems and, more generally, null distributions cannot be derived for classical parametric statistical inference such as hypothesis testing.  This is a fundamental difficulty in doing classical statistics on the space of phylogenetic trees.

We simulated sets of 480 trees with 4 leaves so each tree is a vector in $\R^6$, and randomly selected three trees as vertices for our triangle.  We then drew the tropical line segments between them and computed the Alexandrov curvature.  We repeated this for a large number of triangles and computed the occurrence for each type of curvature we saw as a percentage of the number of triangles sampled (flat, skinny, fat, and undefined).  We repeated this experiment twice, once for 100 sampled triangles and once for 300 sampled triangles; the results are summarized in Table \ref{tab:trees}.

\begin{table}
\begin{center}
\caption{Alexandrov Curvature for Sets of Ultrametric Trees}
\label{tab:trees}       % Give a unique label
\begin{tabular}{lllll}
\hline\noalign{\smallskip}
Sample Size & Flat & Skinny & Fat & Undefined \\
\noalign{\smallskip}\hline\noalign{\smallskip}
100 & 0.07 & 0 & 0.2 & 0.73 \\
300 & 0.1017544 & 0 & 0.1824561 & 0.7157895 \\
\noalign{\smallskip}\hline
\end{tabular}
\end{center}
\end{table}

Here, we see that there are never any skinny triangles, and most triangles have undefined curvature.  There is a significant proportion of fat triangles and some flat triangles.  This indicates that almost certainly, the classical Sturm's algorithm for computing Fr\'{e}chet means that assumes the underlying metric space is nonpositively curved will not be applicable and new techniques will need to be developed.  Moreover, knowing the frequent occurrence of undefined curvature triangle types means that any new method developed to compute Fr\'{e}chet means cannot rely on curvature, as Sturm's algorithm does.

%%%%%%%%%%%%%%%%%%%%%%%%%%%%%%%%%%%%%%%%%%%%%%%%%%%

\section*{Software}

\texttt{R} code to implement all numerical experiments presented in this paper is publicly available
and located on the Tropical Alexandrov GitHub repository at \url{https://github.com/antheamonod/TropAlex}.

%%%%%%%%%%%%%%%%%%%%%%%%%%%%%%%%%%%%%%%%%%%%%%%%%%%

\section{Discussion}
\label{sec:discussion}

In this paper, we studied Alexandrov curvature in $\R^n/\R\one$ with respect to the tropical metric and find wildly varying curvature behavior.  Since Alexandrov curvature is defined in terms of triangles, given that there are five combinatorial types of tropical triangles in the plane, we explored the curvature with respect to each of the five types and find that there exists skinny, fat, flat, and undefined curvature behavior.  We established that in general, Alexandrov curvature is never well-defined for triangle types T2 and T4, while triangle type T3 is always fat in the plane and in higher dimensions.  By way of numerical experiments, we found that type T1 triangles are very often skinny, but can also be undefined, and that type T5 triangles are often undefined but can also be fat.  Our study is novel because it is the first study of Alexandrov curvature in troipcal geometry, and it is also the first study of Alexandrov curvature on combinatorial triangles.

Our results indicate that any computational or statistical methodology to be developed in the tropical projective torus cannot rely on curvature.  This is particularly relevant for statistical studies on the space of phylogenetic trees, which are contained within the tropical projective torus, and where current methodology for statistics does rely on the curvature of the space.  

In addition, our findings are interesting because they show surprising and unexpected geometric behavior, signifying that there is an intricate underlying geometry to the tropical projective torus, despite the fact that it is a well-behaved metric measure space with desirable properties for probability and statistics.  To the best of our knowledge, this is the first occurrence of undefined Alexandrov curvature in such a well-behaved space.  This indicates that there is more to be explored and understood about the tropical projective torus, and invites a study of exactly what are the driving factors influencing the curvature behavior.  This is a complex and challenging problem, because in spite of the polyhedral and combinatorial properties of tropical geometry, there is nevertheless a lack of symmetry and a clear pattern to the behavior.  For example, in the plane, we see that the five types of tropical triangles are defined by the region enclosed by the tropical line segments, and we find that there is one triangle, two quadrangles, one pentagon and one hexagon (see Figure \ref{fig:5types}).  Additionally, as discussed previously, there are various cases for each triangle type: two for each triangle type except for type T2, where there is only one case, and for type T5, where there are six.  This suggests that in order to obtain a complete characterization, each triangle type and case needs to be studied individually, which is difficult given that the number of combinatorial types in an arbitrary dimension is currently unknown.

Our results inspire various future directions for research.  In addition to those that have already been raised throughout the paper, an important question that still remains open is a complete characterization of the regions of an arbitrary tropical triangle that gives rise to each curvature type (positive, negative, zero, and undefined).  Additionally, we would like to determine the ``sizes'' of these regions within the tropical projective torus and what they look like, in terms of their geometry. Our work, including results and simulations, helps to understand the intricate behavior of Alexandrov curvature, but a full description still needs to be achieved.

From a statistical and probabilistic viewpoint, we would also like to understand how probability measures influence the proportions of observed triangle types and curvature behavior. As we have seen in our work, the distribution plays an important role, since we observe the appearance of flat triangles via uniform sampling with integer coordinates.  A particular convergence question we would like to fully understand is the tendency towards triangles of undefined curvature as the dimension $n$ of the tropical projective torus grows; we would also like to be able to quantify this tendency.

%%%%%%%%%%%%%%%%%%%%%%%%%%%%%%%%%%%%%%%%%%%%%%%%%%%

\section*{Acknowledgments}

We would like to thank Yueqi Cao, Stephan Huckemann, Bo Lin, and Ruriko Yoshida for discussions on our work.  We are particularly grateful to Emil Saucan for his many helpful insights on curvature and to Beatrice Matteo for her help with the simulations presented.  We also wish to thank anonymous reviewers whose helpful comments have greatly improved our work.

%%%%%%%%%%%%%%%%%%%%%%%%%%%%%%%%%%%%%%%%%%%%%%%%%%%

%\appendix
%\renewcommand{\thesection}{Appendix}
%\renewcommand{\thesubsection}{A\arabic{subsection}}
%\renewcommand{\theAppDefinition}{A\arabic{AppDefinition}}
%\renewcommand{\theAppClaim}{A\arabic{AppClaim}}
%\section*{Appendix}

%%%%%%%%%%%%%%%%%%%%%%%%%%%%%%%%%%%%%%%%%%%%%%%%%%%

%\newpage
%\section*{Figures}

%%%%%%%%%%%%%%%%%%%%%%%%%%%%%%%%%%%%%%%%%%%%%%%%%%%
%%%%%%%%%%%%%%%%%%%%%%%%%%%%%%%%%%%%%%%%%%%%%%%%%%%
%%%%%%%%%%%%%%%%%%%%%%%%%%%%%%%%%%%%%%%%%%%%%%%%%%%

% \section{Referencing}

% The bibliography file (a standard {\em .bib} formatted file) is % pulled in at the end, and articles, papers, books,
% web sites etc listed there are cited in various ways, such as \cite{carvalho2008bfrm} and \cite{aguilar2000jbes},
%as well as \citep{carvalho2008bfrm} and \citep{aguilar2000jbes}, and %also
%\citep{carvalho2008bfrm,aguilar2000jbes,west2003bayes7}. You can also use (\citet{aguilar2000jbes}) or
%(\citet{carvalho2008bfrm,aguilar2000jbes,west2003bayes7})  at a more hands-on level.

%\clearpage
%\newpage
\bibliographystyle{chicago}  % or choose another bib list style
\bibliography{Birthday_ref} % edit bibexample.bib file ...

\begin{thebibliography}{}

\bibitem[\protect\citeauthoryear{Akian, Gaubert, Ni\c{t}ic\u{a}, and
  Singer}{Akian et~al.}{2011}]{AKIAN20113261}
Akian, M., S.~Gaubert, V.~Ni\c{t}ic\u{a}, and I.~Singer (2011).
\newblock Best approximation in max-plus semimodules.
\newblock {\em Linear Algebra and its Applications\/}~{\em 435\/}(12),
  3261--3296.

\bibitem[\protect\citeauthoryear{Alessandrini}{Alessandrini}{2013}]{Alessandrini+2013+155+190}
Alessandrini, D. (2013).
\newblock Logarithmic limit sets of real semi-algebraic sets.
\newblock ~{\em 13\/}(1), 155--190.

\bibitem[\protect\citeauthoryear{Allamigeon, Benchimol, Gaubert, and
  Joswig}{Allamigeon et~al.}{2018}]{allamigeon2018log}
Allamigeon, X., P.~Benchimol, S.~Gaubert, and M.~Joswig (2018).
\newblock Log-barrier interior point methods are not strongly polynomial.
\newblock {\em SIAM Journal on Applied Algebra and Geometry\/}~{\em 2\/}(1),
  140--178.

\bibitem[\protect\citeauthoryear{Ansola and de~la Puente}{Ansola and de~la
  Puente}{2009}]{ansola2009note}
Ansola, M. and M.~J. de~la Puente (2009).
\newblock A note on tropical triangles in the plane.
\newblock {\em Acta Mathematica Sinica, English Series\/}~{\em 25\/}(11), 1775.

\bibitem[\protect\citeauthoryear{Aurenhammer, Klein, and Lee}{Aurenhammer
  et~al.}{2013}]{aurenhammer2013voronoi}
Aurenhammer, F., R.~Klein, and D.-T. Lee (2013).
\newblock {\em Voronoi diagrams and {D}elaunay triangulations}.
\newblock World Scientific Publishing Company.

\bibitem[\protect\citeauthoryear{Banchoff}{Banchoff}{1970}]{banchoff1970critical}
Banchoff, T.~F. (1970).
\newblock Critical points and curvature for embedded polyhedral surfaces.
\newblock {\em The American Mathematical Monthly\/}~{\em 77\/}(5), 475--485.

\bibitem[\protect\citeauthoryear{Boneh and Golan}{Boneh and
  Golan}{1979}]{boneh1979constraints}
Boneh, A. and A.~Golan (1979).
\newblock Constraints’ redundancy and feasible region boundedness by random
  feasible point generator (rfpg).
\newblock In {\em Third European congress on operations research (EURO III),
  Amsterdam}.

\bibitem[\protect\citeauthoryear{Bridson and Haefliger}{Bridson and
  Haefliger}{2013}]{bridson2013metric}
Bridson, M.~R. and A.~Haefliger (2013).
\newblock {\em Metric spaces of non-positive curvature}, Volume 319.
\newblock Springer Science \& Business Media.

\bibitem[\protect\citeauthoryear{Burago, Burago, Burago, Ivanov, Ivanov, and
  Ivanov}{Burago et~al.}{2001}]{burago2001course}
Burago, D., I.~D. Burago, Y.~Burago, S.~Ivanov, S.~V. Ivanov, and S.~A. Ivanov
  (2001).
\newblock {\em A course in metric geometry}, Volume~33.
\newblock American Mathematical Soc.

\bibitem[\protect\citeauthoryear{Cartwright, H{\"a}bich, Sturmfels, and
  Werner}{Cartwright et~al.}{2011}]{cartwright2011mustafin}
Cartwright, D., M.~H{\"a}bich, B.~Sturmfels, and A.~Werner (2011).
\newblock Mustafin varieties.
\newblock {\em Selecta Mathematica\/}~{\em 17\/}(4), 757--793.

\bibitem[\protect\citeauthoryear{Cheeger}{Cheeger}{2015}]{cheeger2015lower}
Cheeger, J. (2015).
\newblock A lower bound for the smallest eigenvalue of the laplacian.
\newblock In {\em Problems in Analysis}, pp.\  195--200. Princeton University
  Press.

\bibitem[\protect\citeauthoryear{Cheeger, M{\"u}ller, and Schrader}{Cheeger
  et~al.}{1984}]{cheeger1984curvature}
Cheeger, J., W.~M{\"u}ller, and R.~Schrader (1984).
\newblock On the curvature of piecewise flat spaces.
\newblock {\em Communications in mathematical Physics\/}~{\em 92\/}(3),
  405--454.

\bibitem[\protect\citeauthoryear{Cohen, Gaubert, and Quadrat}{Cohen
  et~al.}{2004}]{COHEN2004395}
Cohen, G., S.~Gaubert, and J.-P. Quadrat (2004).
\newblock Duality and separation theorems in idempotent semimodules.
\newblock {\em Linear Algebra and its Applications\/}~{\em 379}, 395--422.
\newblock Special Issue on the Tenth ILAS Conference (Auburn, 2002).

\bibitem[\protect\citeauthoryear{Criado, Joswig, and Santos}{Criado
  et~al.}{2021}]{criado2021tropical}
Criado, F., M.~Joswig, and F.~Santos (2021).
\newblock Tropical bisectors and {V}oronoi diagrams.
\newblock {\em Foundations of Computational Mathematics\/}, 1--38.

\bibitem[\protect\citeauthoryear{Develin and Sturmfels}{Develin and
  Sturmfels}{2004}]{develin2004tropical}
Develin, M. and B.~Sturmfels (2004).
\newblock Tropical convexity.
\newblock {\em Documenta Mathematica\/}~{\em 9}, 1--27.

\bibitem[\protect\citeauthoryear{Gromov}{Gromov}{1987}]{gromov1987hyperbolic}
Gromov, M. (1987).
\newblock Hyperbolic groups.
\newblock In {\em Essays in group theory}, pp.\  75--263. Springer.

\bibitem[\protect\citeauthoryear{Gromov}{Gromov}{2019}]{gromov2019four}
Gromov, M. (2019).
\newblock Four lectures on scalar curvature.
\newblock {\em arXiv:1908.10612\/}.

\bibitem[\protect\citeauthoryear{Hampe}{Hampe}{2015}]{hampe2015tropical}
Hampe, S. (2015).
\newblock Tropical linear spaces and tropical convexity.
\newblock {\em The Electronic Journal of Combinatorics\/}~{\em 22\/}(4),
  P4--43.

\bibitem[\protect\citeauthoryear{Hotz, Skwerer, Huckemann, Le, Marron,
  Mattingly, Miller, Nolen, Owen, and Patrangenaru}{Hotz
  et~al.}{2013}]{10.1214/12-AAP899}
Hotz, T., S.~Skwerer, S.~Huckemann, H.~Le, J.~S. Marron, J.~C. Mattingly,
  E.~Miller, J.~Nolen, M.~Owen, and V.~Patrangenaru (2013).
\newblock {Sticky central limit theorems on open books}.
\newblock {\em The Annals of Applied Probability\/}~{\em 23\/}(6), 2238 --
  2258.

\bibitem[\protect\citeauthoryear{Huckemann and Eltzner}{Huckemann and
  Eltzner}{2021}]{https://doi.org/10.1002/wics.1526}
Huckemann, S.~F. and B.~Eltzner (2021).
\newblock Data analysis on nonstandard spaces.
\newblock {\em WIREs Computational Statistics\/}~{\em 13\/}(3), e1526.

\bibitem[\protect\citeauthoryear{Ishida}{Ishida}{1990}]{ishida1990pseudo}
Ishida, M. (1990).
\newblock Pseudo-curvature of a graph.
\newblock In {\em Lecture at ‘Workshop on topological graph theory'}.

\bibitem[\protect\citeauthoryear{Jell, Scheiderer, and Yu}{Jell
  et~al.}{2020}]{10.1093/imrn/rnaa112}
Jell, P., C.~Scheiderer, and J.~Yu (2020).
\newblock {Real tropicalization and analytification of semialgebraic Sets}.
\newblock {\em International Mathematics Research Notices\/}.
\newblock rnaa112.

\bibitem[\protect\citeauthoryear{Jost}{Jost}{2012}]{jost2012nonpositive}
Jost, J. (2012).
\newblock {\em Nonpositive curvature: geometric and analytic aspects}.
\newblock Birkh{\"a}user.

\bibitem[\protect\citeauthoryear{Joswig, Sturmfels, and Yu}{Joswig
  et~al.}{2007}]{joswig2007josephine}
Joswig, M., B.~Sturmfels, and J.~Yu (2007).
\newblock Affine buildings and tropical convexity.
\newblock {\em Albanian J. Math.\/}~{\em 1\/}(4), 187--211.

\bibitem[\protect\citeauthoryear{Kobayashi and Wynn}{Kobayashi and
  Wynn}{2019}]{Kobayashi2019}
Kobayashi, K. and H.~P. Wynn (2019).
\newblock Empirical geodesic graphs and {CAT}(k) metrics for data analysis.
\newblock {\em Statistics and Computing\/}~{\em 30\/}(1), 1--18.

\bibitem[\protect\citeauthoryear{Kolokoltsov and Maslov}{Kolokoltsov and
  Maslov}{2013}]{kolokoltsov2013idempotent}
Kolokoltsov, V.~N. and V.~P. Maslov (2013).
\newblock {\em Idempotent analysis and its applications}, Volume 401.
\newblock Springer Science \& Business Media.

\bibitem[\protect\citeauthoryear{Lee, Li, Lin, and Monod}{Lee
  et~al.}{2021}]{Lee2021}
Lee, W., W.~Li, B.~Lin, and A.~Monod (2021).
\newblock Tropical optimal transport and {W}asserstein distances.
\newblock {\em Information Geometry\/}.

\bibitem[\protect\citeauthoryear{Lin, Monod, and Yoshida}{Lin
  et~al.}{2022}]{lin2022tropical}
Lin, B., A.~Monod, and R.~Yoshida (2022).
\newblock Tropical geometric variation of tree shapes.
\newblock {\em Discrete \& Computational Geometry\/}~{\em 68\/}(3), 817--849.

\bibitem[\protect\citeauthoryear{Lin, Sturmfels, Tang, and Yoshida}{Lin
  et~al.}{2017}]{doi:10.1137/16M1079841}
Lin, B., B.~Sturmfels, X.~Tang, and R.~Yoshida (2017).
\newblock Convexity in tree spaces.
\newblock {\em SIAM Journal on Discrete Mathematics\/}~{\em 31\/}(3),
  2015--2038.

\bibitem[\protect\citeauthoryear{Loho and Schymura}{Loho and
  Schymura}{2020}]{Loho2020}
Loho, G. and M.~Schymura (2020, September).
\newblock Tropical {E}hrhart theory and tropical volume.
\newblock {\em Research in the Mathematical Sciences\/}~{\em 7\/}(4).

\bibitem[\protect\citeauthoryear{Maclagan and Sturmfels}{Maclagan and
  Sturmfels}{2015}]{maclagan2015introduction}
Maclagan, D. and B.~Sturmfels (2015).
\newblock {\em Introduction to Tropical Geometry (Graduate Studies in
  Mathematics)}.
\newblock American Mathematical Society.

\bibitem[\protect\citeauthoryear{Maragos, Charisopoulos, and Theodosis}{Maragos
  et~al.}{2021}]{9394420}
Maragos, P., V.~Charisopoulos, and E.~Theodosis (2021).
\newblock Tropical geometry and machine learning.
\newblock {\em Proceedings of the IEEE\/}~{\em 109\/}(5), 728--755.

\bibitem[\protect\citeauthoryear{Monod, Lin, Yoshida, and Kang}{Monod
  et~al.}{2018}]{monod2018tropical}
Monod, A., B.~Lin, R.~Yoshida, and Q.~Kang (2018).
\newblock Tropical {G}eometry of {P}hylogenetic {T}ree {S}pace: {A}
  {S}tatistical {P}erspective.
\newblock {\em arXiv:1805.12400\/}.

\bibitem[\protect\citeauthoryear{Najman and Romon}{Najman and
  Romon}{2014}]{najman2014discrete}
Najman, L. and P.~Romon (2014).
\newblock Discrete curvature: {T}heory and applications.

\bibitem[\protect\citeauthoryear{Najman and Romon}{Najman and
  Romon}{2017}]{najman2017modern}
Najman, L. and P.~Romon (2017).
\newblock {\em Modern approaches to discrete curvature}, Volume 2184.
\newblock Springer.

\bibitem[\protect\citeauthoryear{Newton}{Newton}{1833}]{newton1833philosophiae}
Newton, I. (1833).
\newblock {\em Philosophiae naturalis principia mathematica}, Volume~2.
\newblock typis A. et JM Duncan.

\bibitem[\protect\citeauthoryear{Ohta}{Ohta}{2012}]{ohta2012barycenters}
Ohta, S.-I. (2012).
\newblock Barycenters in {A}lexandrov spaces of curvature bounded below.
\newblock {\em Advances in Geometry\/}~{\em 12\/}(4), 571--587.

\bibitem[\protect\citeauthoryear{Ollivier}{Ollivier}{2007}]{OLLIVIER2007643}
Ollivier, Y. (2007).
\newblock Ricci curvature of metric spaces.
\newblock {\em Comptes Rendus Mathematique\/}~{\em 345\/}(11), 643--646.

\bibitem[\protect\citeauthoryear{Ollivier}{Ollivier}{2009}]{ollivier2009ricci}
Ollivier, Y. (2009).
\newblock Ricci curvature of markov chains on metric spaces.
\newblock {\em Journal of Functional Analysis\/}~{\em 256\/}(3), 810--864.

\bibitem[\protect\citeauthoryear{Ollivier}{Ollivier}{2011}]{ollivier2011visual}
Ollivier, Y. (2011).
\newblock A visual introduction to {R}iemannian curvatures and some discrete
  generalizations.
\newblock {\em Analysis and Geometry of Metric Measure Spaces: Lecture Notes of
  the 50th S{\'e}minaire de Math{\'e}matiques Sup{\'e}rieures (SMS),
  Montr{\'e}al\/}~{\em 56}, 197--219.

\bibitem[\protect\citeauthoryear{Pachter and Sturmfels}{Pachter and
  Sturmfels}{2005}]{ASCB}
Pachter, L. and B.~Sturmfels (2005).
\newblock {\em Algebraic statistics for computational biology}, Volume~13.
\newblock Cambridge University Press.

\bibitem[\protect\citeauthoryear{Plaut}{Plaut}{2001}]{plaut2001metric}
Plaut, C. (2001).
\newblock Metric spaces of curvature $\geq k$, {C}hapter 16.
\newblock {\em Handbook of Geometric Topology\/}.

\bibitem[\protect\citeauthoryear{Saucan}{Saucan}{2006}]{saucan2006curvature}
Saucan, E. (2006).
\newblock Curvature--smooth, piecewise-linear and metric.
\newblock {\em What is Geometry\/}, 237--268.

\bibitem[\protect\citeauthoryear{Smith}{Smith}{1980}]{smith1980monte}
Smith, R.~L. (1980).
\newblock Monte carlo procedures for generating random feasible solutions to
  mathematical programs.
\newblock In {\em A Bulletin of the ORSA/TIMS Joint National Meeting,
  Washington, DC}, Volume 101.

\bibitem[\protect\citeauthoryear{Speyer and Sturmfels}{Speyer and
  Sturmfels}{2004}]{Speyer2004}
Speyer, D. and B.~Sturmfels (2004).
\newblock The tropical {G}rassmannian.
\newblock {\em Advances in Geometry\/}~{\em 4\/}(3).

\bibitem[\protect\citeauthoryear{Spivak}{Spivak}{1973}]{spivak1973comprehensive}
Spivak, M. (1973).
\newblock A comprehensive introduction to differential geometry.
\newblock {\em Bull. Amer. Math. Soc\/}~{\em 79}, 303--306.

\bibitem[\protect\citeauthoryear{Sturm, Coulhon, Grigor\'{y}an, et~al.}{Sturm
  et~al.}{2003}]{sturm2003probability}
Sturm, K.-T., T.~Coulhon, A.~Grigor\'{y}an, et~al. (2003).
\newblock Probability measures on metric spaces of nonpositive.
\newblock {\em Lecture Notes from a Quarter Program on Heat Kernels, Random
  Walks, and Analysis on Manifolds and Graphs: April 16-July 13, 2002, Emile
  Borel Centre of the Henri Poincar{\'e} Institute, Paris, France\/}~{\em 338}.

\bibitem[\protect\citeauthoryear{Tervonen, {van Valkenhoef}, Baştürk, and
  Postmus}{Tervonen et~al.}{2013}]{TERVONEN2013552}
Tervonen, T., G.~{van Valkenhoef}, N.~Baştürk, and D.~Postmus (2013).
\newblock Hit-and-run enables efficient weight generation for simulation-based
  multiple criteria decision analysis.
\newblock {\em European Journal of Operational Research\/}~{\em 224\/}(3),
  552--559.

\bibitem[\protect\citeauthoryear{Tran}{Tran}{2020}]{tran2020tropical}
Tran, N.~M. (2020).
\newblock Tropical {G}aussians: a brief survey.
\newblock {\em Algebraic Statistics\/}~{\em 11\/}(2), 155--168.

\bibitem[\protect\citeauthoryear{Trimmel, Petzka, and Sminchisescu}{Trimmel
  et~al.}{2021}]{trimmel2021tropex}
Trimmel, M., H.~Petzka, and C.~Sminchisescu (2021).
\newblock Tropex: An algorithm for extracting linear terms in deep neural
  networks.
\newblock In {\em International Conference on Learning Representations}.

\bibitem[\protect\citeauthoryear{Yoshida}{Yoshida}{2021}]{math9070779}
Yoshida, R. (2021).
\newblock Tropical balls and its applications to {K} nearest neighbor over the
  space of phylogenetic trees.
\newblock {\em Mathematics\/}~{\em 9\/}(7).

\bibitem[\protect\citeauthoryear{Yoshida and Cox}{Yoshida and
  Cox}{2022}]{yoshida2022tree}
Yoshida, R. and S.~Cox (2022).
\newblock Tree topologies along a tropical line segment.
\newblock {\em Vietnam Journal of Mathematics\/}, 1--25.

\bibitem[\protect\citeauthoryear{Yoshida, Takamori, Matsumoto, and
  Miura}{Yoshida et~al.}{2023}]{yoshida2023tropical}
Yoshida, R., M.~Takamori, H.~Matsumoto, and K.~Miura (2023).
\newblock Tropical support vector machines: Evaluations and extension to
  function spaces.
\newblock {\em Neural Networks\/}~{\em 157}, 77--89.

\bibitem[\protect\citeauthoryear{Yoshida, Zhang, and Zhang}{Yoshida
  et~al.}{2019}]{yoshida2019tropical}
Yoshida, R., L.~Zhang, and X.~Zhang (2019).
\newblock Tropical principal component analysis and its application to
  phylogenetics.
\newblock {\em Bulletin of Mathematical Biology\/}~{\em 81\/}(2), 568--597.

\bibitem[\protect\citeauthoryear{Zhang, Naitzat, and Lim}{Zhang
  et~al.}{2018}]{pmlr-v80-zhang18i}
Zhang, L., G.~Naitzat, and L.-H. Lim (2018, 10--15 Jul).
\newblock Tropical geometry of deep neural networks.
\newblock In J.~Dy and A.~Krause (Eds.), {\em Proceedings of the 35th
  International Conference on Machine Learning}, Volume~80 of {\em Proceedings
  of Machine Learning Research}, pp.\  5824--5832. PMLR.

\end{thebibliography}

%%%%%%%%%%%%%%%%%%%%%%%%%%%%%%%%%%%%%%%%%%%%%%%%%%%
%%%%%%%%%%%%%%%%%%%%%%%%%%%%%%%%%%%%%%%%%%%%%%%%%%%
%%%%%%%%%%%%%%%%%%%%%%%%%%%%%%%%%%%%%%%%%%%%%%%%%%%

\end{document}